\documentclass[letterpaper,11pt]{article}

\usepackage{xspace,exscale,relsize}
\usepackage{fancybox,shadow}

\usepackage{color}
\usepackage{amsfonts}

\usepackage[title]{appendix}

\usepackage{mathtools}

\makeatletter
\let\over=\@@over \let\overwithdelims=\@@overwithdelims
\let\atop=\@@atop \let\atopwithdelims=\@@atopwithdelims
\let\above=\@@above \let\abovewithdelims=\@@abovewithdelims
\makeatother
\usepackage{rotating}

\usepackage{float}
\usepackage{algorithm}
\usepackage{algpseudocode}

\usepackage{ifpdf}

\usepackage{psfrag}

\usepackage{prettyref,enumerate}

\usepackage{tikz}
\usetikzlibrary{arrows}
\tikzstyle{int}=[draw, fill=blue!20, minimum size=2em]
\tikzstyle{dot}=[circle, draw, fill=blue!20, minimum size=2em]
\tikzstyle{init} = [pin edge={to-,thin,black}]

\usepackage[all]{xy}

\newcommand{\vf}{\boldsymbol{f}}
\newcommand{\vferm}{\hat{\boldsymbol{f}}_{\mathsf{erm}}}
\newcommand{\calvF}{\boldsymbol{\mathcal F}}

\newcommand{\vfbest}{\boldsymbol{f}^*}
\newcommand{\vX}{\boldsymbol{X}}

\newcommand{\vx}{\boldsymbol{x}}

\newcommand{\ve}{\boldsymbol{e}}
\newcommand{\vtheta}{\boldsymbol{\theta}}

\ifx\eqref\undefined
\newcommand{\eqref}[1]{~(\ref{#1})}
\fi
\ifx\mod\undefined
\def\mod{\mathop{\rm mod}}
\fi

\usepackage{bm}
\def\vect#1{\bm{#1}}

\newcommand{\norm}[1]{{\left\Vert #1 \right\Vert}}

\newcommand{\argmin}{\mathop{\rm argmin}}
\newcommand{\argmax}{\mathop{\rm argmax}}

\def\exp{\mathop{\rm exp}}

\def\EE{\Expect}

\def\PP{\mathbb{P}}

\def\eqdef{\triangleq}

\def\simiid{\stackrel{iid}{\sim}}

\newcommand{\abs}[1]{\left| #1 \right|}

\makeatletter
\def\bbordermatrix#1{\begingroup \m@th
	\@tempdima 4.75\p@
	\setbox\z@\vbox{%
		\def\cr{\crcr\noalign{\kern2\p@\global\let\cr\endline}}%
		\ialign{$##$\hfil\kern2\p@\kern\@tempdima&\thinspace\hfil$##$\hfil
			&&\quad\hfil$##$\hfil\crcr
			\omit\strut\hfil\crcr\noalign{\kern-\baselineskip}%
			#1\crcr\omit\strut\cr}}%
	\setbox\tw@\vbox{\unvcopy\z@\global\setbox\@ne\lastbox}%
	\setbox\tw@\hbox{\unhbox\@ne\unskip\global\setbox\@ne\lastbox}%
	\setbox\tw@\hbox{$\kern\wd\@ne\kern-\@tempdima\left[\kern-\wd\@ne
		\global\setbox\@ne\vbox{\box\@ne\kern2\p@}%
		\vcenter{\kern-\ht\@ne\unvbox\z@\kern-\baselineskip}\,\right]$}%
	\null\;\vbox{\kern\ht\@ne\box\tw@}\endgroup}
\makeatother

\newcommand{\stepa}[1]{\overset{\rm (a)}{#1}}
\newcommand{\stepb}[1]{\overset{\rm (b)}{#1}}
\newcommand{\stepc}[1]{\overset{\rm (c)}{#1}}

\newcommand{\btheta}{\vect{\theta}}
\newcommand{\Poi}{\mathrm{Poi}}

\newcommand{\reals}{\mathbb{R}}

\newcommand{\integers}{\mathbb{Z}}

\newcommand{\Expect}{\mathbb{E}}

\newcommand{\iid}{iid\xspace}
\newcommand{\ind}{ind.\xspace}

\newcommand{\pth}[1]{\left( #1 \right)}
\newcommand{\qth}[1]{\left[ #1 \right]}
\newcommand{\sth}[1]{\left\{ #1 \right\}}

\newcommand{\iiddistr}{{\stackrel{\text{\iid}}{\sim}}}
\newcommand{\inddistr}{{\stackrel{\text{\ind}}{\sim}}}

\newcommand{\Binom}{\text{Binom}}

\newcommand{\indc}[1]{{\mathbf{1}_{\left\{{#1}\right\}}}}

\definecolor{myblue}{rgb}{.8, .8, 1}
\definecolor{mathblue}{rgb}{0.2472, 0.24, 0.6} 
\definecolor{mathred}{rgb}{0.6, 0.24, 0.442893}
\definecolor{mathyellow}{rgb}{0.6, 0.547014, 0.24}

\newcommand{\sfR}{{\mathsf{R}}}

\newcommand{\calF}{{\mathcal{F}}}

\newcommand{\calP}{{\mathcal{P}}}

\newrefformat{eq}{(\ref{#1})}
\newrefformat{thm}{Theorem~\ref{#1}}
\newrefformat{th}{Theorem~\ref{#1}}
\newrefformat{chap}{Chapter~\ref{#1}}
\newrefformat{sec}{Section~\ref{#1}}
\newrefformat{seca}{Section~\ref{#1}}
\newrefformat{algo}{Algorithm~\ref{#1}}
\newrefformat{fig}{Fig.~\ref{#1}}
\newrefformat{tab}{Table~\ref{#1}}
\newrefformat{rmk}{Remark~\ref{#1}}
\newrefformat{clm}{Claim~\ref{#1}}
\newrefformat{def}{Definition~\ref{#1}}
\newrefformat{cor}{Corollary~\ref{#1}}
\newrefformat{lmm}{Lemma~\ref{#1}}
\newrefformat{prop}{Proposition~\ref{#1}}
\newrefformat{pr}{Proposition~\ref{#1}}
\newrefformat{property}{Property~\ref{#1}}
\newrefformat{app}{Appendix~\ref{#1}}
\newrefformat{apx}{Appendix~\ref{#1}}
\newrefformat{ex}{Example~\ref{#1}}
\newrefformat{exer}{Exercise~\ref{#1}}
\newrefformat{soln}{Solution~\ref{#1}}

\def\unifto{\mathop{{\mskip 3mu plus 2mu minus 1mu%
			\setbox0=\hbox{$\mathchar"3221$}%
			\raise.6ex\copy0\kern-\wd0%
			\lower0.5ex\hbox{$\mathchar"3221$}}\mskip 3mu plus 2mu minus 1mu}}

\ifx\lesssim\undefined
\def\simleq{{{\mskip 3mu plus 2mu minus 1mu%
			\setbox0=\hbox{$\mathchar"013C$}%
			\raise.2ex\copy0\kern-\wd0%
			\lower0.9ex\hbox{$\mathchar"0218$}}\mskip 3mu plus 2mu minus 1mu}}
\else
\def\simleq{\lesssim}
\fi

\ifx\gtrsim\undefined
\def\simgeq{{{\mskip 3mu plus 2mu minus 1mu%
			\setbox0=\hbox{$\mathchar"013E$}%
			\raise.2ex\copy0\kern-\wd0%
			\lower0.9ex\hbox{$\mathchar"0218$}}\mskip 3mu plus 2mu minus 1mu}}
\else
\def\simgeq{\gtrsim}
\fi

\newif\ifmapx
{\catcode`/=0 \catcode`\\=12/gdef/mkillslash\#1{#1}}
\edef\jobnametmp{\expandafter\string\csname ic_apx\endcsname}
\edef\jobnameapx{\expandafter\mkillslash\jobnametmp}
\edef\jobnameexpand{\jobname}
\ifx\jobnameexpand\jobnameapx
\mapxtrue
\else
\mapxfalse
\fi

\newcommand{\polylog}{\mathsf{polylog}}

\newcommand{\fbest}{f^{*}}
\newcommand{\ferm}{\hat f_\mathsf{erm}}
\newcommand{\frob}{\hat f_{\mathsf{Rob}}}

\newcommand{\subexpo}{\mathsf{SubE}}
\newcommand{\Regret}{\mathsf{Regret}}

\renewcommand{\hat}{\widehat}

\newcommand{\bbR}{\mathbb R}

\newcommand{\bbZ}{\mathbb Z}

\newcommand{\bbE}{\mathbb E}
\newcommand{\bbP}{\mathbb P}

\usepackage{algorithm}
\usepackage{algpseudocode}
\usepackage[margin=1in]{geometry}  
\usepackage{graphicx}
\usepackage{amsthm}
\usepackage{amssymb}
\usepackage{authblk}

\usepackage{extarrows}

\usepackage{amsmath,mathrsfs}
\makeatletter
\let\over=\@@over \let\overwithdelims=\@@overwithdelims
\let\atop=\@@atop \let\atopwithdelims=\@@atopwithdelims
\let\above=\@@above \let\abovewithdelims=\@@abovewithdelims
\makeatother
\usepackage{rotating}

\usepackage{ifpdf}

\usepackage{subfigure}
\usepackage{psfrag}

\usepackage{prettyref}
\usepackage{enumitem}

\usepackage{tikz}
\usetikzlibrary{arrows}
\tikzstyle{int}=[draw, fill=blue!20, minimum size=2em]
\tikzstyle{dot}=[circle, draw, fill=blue!20, minimum size=2em]
\tikzstyle{init} = [pin edge={to-,thin,black}]

\usepackage[
colorlinks=true,
citecolor=red,
linkcolor=blue,
anchorcolor=red,
urlcolor=blue,
pdfauthor={Soham Jana}
]{hyperref}

\usepackage[all]{xy}

\usepackage{mathtools}

\usepackage{ifthen}
\newboolean{aos}
\setboolean{aos}{FALSE}

\usepackage{amssymb, amsthm}
\newtheorem{theorem}{Theorem}
\newtheorem{lemma}{Lemma}

\newtheorem{remark}{Remark}

\begin{document}
	
\ifpdf
	\DeclareGraphicsExtensions{.pgf,.jpg}
	\graphicspath{{figures/}{plots/}}
	\fi

    \title{Empirical Bayes via ERM and Rademacher complexities: the Poisson
model}

\author{Soham Jana, Yury Polyanskiy, Anzo Teh and Yihong Wu\thanks{
		S.J. is with the Department of Operations Research and Financial Engineering, Princeton University, Princeton, NJ, email: \url{soham.jana@princeton.edu}.
		Y.P. and A.T. are with the Department of EECS, MIT, Cambridge,
		MA, email: \url{yp@mit.edu} and \url{anzoteh@mit.edu}. Y.W. is with the Department of Statistics and Data Science, Yale
		University, New Haven, CT, email: \url{yihong.wu@yale.edu}.}}

	\maketitle

\begin{abstract}%

We consider the problem of empirical Bayes estimation for (multivariate) Poisson means. Existing solutions that have been shown theoretically optimal for minimizing the regret (excess risk over the Bayesian oracle that knows the prior) have several shortcomings. For example, the classical Robbins estimator does not retain the monotonicity property of the Bayes estimator and performs poorly under moderate sample size. Estimators based on the minimum distance and non-parametric maximum likelihood (NPMLE) methods  correct these issues, but are computationally expensive with complexity growing exponentially with dimension. Extending the approach of 
\cite{barbehenn2022nonparametric}, in this work we construct monotone estimators based on empirical risk minimization (ERM) that retain similar theoretical guarantees and can be computed much more efficiently. Adapting the idea of offset Rademacher complexity~\cite{liang2015learning} to the non-standard loss and function class in empirical Bayes, we show that the shape-constrained ERM estimator attains the minimax regret within constant factors in one dimension and within logarithmic factors in multiple dimensions. 

\end{abstract}
 
\tableofcontents



\section{Introduction}

At the heart of modern large-scale inference \cite{efron2012large}, empirical Bayes is a classical topic and powerful formalism in statistics and machine learning. Consider the Poisson model in one dimension as a concrete example. 
In a Bayesian setting, the latent parameter $\theta$ is drawn from a prior $\pi$ and the observation $X$ is then sampled from $\Poi(\theta)$, the Poisson distribution with mean $\theta$. In other words, $X$ is distributed according to the following Poisson mixture $p_\pi$ with mixing distribution $\pi$:
\begin{align}
 p_\pi(x)=\int{e^{-\theta}{\theta^x\over x!}}d\pi(\theta), \quad x\in\integers_+.
\end{align}
The Bayes estimator for $\theta$ that minimizes the squared error is the posterior mean, which can be expressed in terms of the mixture density as follows:
\begin{align}
	\fbest(x)=(x+1){p_\pi(x+1)\over p_\pi(x)}
 \label{eq:bayes}
\end{align}

In the empirical Bayes setting, the prior $\pi$ is unknown but we have access to a training sample $X_1,\dots,X_n$ drawn independently from the mixture $ p_\pi$. The goal is to learn a data-driven rule that produces vanishing excess risk over the Bayes risk, known as the {\it regret}\footnote{
In the literature there are multiple ways to formulate the regret in empirical Bayes estimation \cite{zhang2003compound}. As opposed to the formulation (known as the individual regret)
in \prettyref{eq:regret}, where the data are split into the  training set  $X_1,\ldots,X_n$ and the test set $X$, one can consider the total excess risk of estimating the latent parameters $\theta_1,\ldots,\theta_n$ based on $X_1,\ldots,X_n$ over the Bayes risk. This quantity, known as the total regret, in fact equals to $n$ times the individual regret \prettyref{eq:regret} (with $n$ replaced by $n-1$) as shown in \cite[Lemma 5]{polyanskiy2021sharp}.
}
\begin{align}
\label{eq:regret}
	\Regret_\pi(f) \eqdef
 \EE\qth{(\hat f(X)-\theta)^2}-\EE\qth{(\fbest(X)-\theta)^2}.
\end{align}
The problem of interest in this context is thus:
\begin{quote}
    {\it Can we construct computationally efficient and practically sound estimators of $\fbest$ with optimal regret over a class of priors?}
\end{quote}

Preliminary analyses of the Poisson empirical Bayes problem go back to \cite{Rob51,Rob56}, who proposed the following rule as an empirical approximation of \prettyref{eq:bayes}:
\begin{align}\label{eq:est-robbins}
	\frob(X)\eqdef \frob(X;X_1,\dots,X_n)=(X+1){N_n(X+1)\over N_n(X)+1}
\end{align}
where $N_n(x)= \sum_{i=1}^n\indc{X_i=x}$ is the empirical count for each $x\in \integers_+$ in the training sample.
Such an approach is termed ``$f$-modeling'' that focuses on approximating  the mixture density  \cite{efron2014two}. Recent theoretical developments \cite{BGR13,PW20} have established that the Robbins method achieves the optimal rate of regret when $\pi$ has either bounded support or subexponential tails. On the other hand, in practice, it is well-recognized that the Robbins estimator suffers from multiple shortcomings such as numerical instability (cf.~e.g.~\cite[Section 1]{maritz1968smooth}, \cite[Section 1.9]{maritz2018empirical}, \cite[Section 6.1]{efron2021computer}) and lack of regularity properties, including, notably, the desired monotonicity property of the Bayes rule $\fbest$ (see \cite{HS83}).

In another approach to the empirical Bayes problem, known as ``$g$-modeling'' \cite{efron2014two},
one tries to mimic the structure of the Bayes estimator by substituting the prior in the posterior mean with a suitable estimator. 
It has recently been shown that optimal regret can be attained by $g$-modeling estimators based on the minimum distance methodology that first finds the best approximation $p_{\hat \pi}$ to the empirical distribution of the training data under suitable distances then applies the Bayes rule with the learned prior $\hat \pi$.
A prominent example is the
nonparametric maximum likelihood estimator (NPMLE) 
\begin{equation}
\hat \pi_{\sf NPMLE} = \argmax_Q \prod_{i=1}^n p_Q(X_i)
    \label{eq:NPMLE}
\end{equation}
which minimizes the Kullback-Leibler divergence. 
Thanks to their Bayesian form, these estimators inherit the desired regularity of Bayes estimator (such as monotonicity) and lead to  more stable, accurate, and interpretable estimates in practice. 
Recently, \cite{jana2022optimal} has shown that a suite of minimum-distance estimators, including the NPMLE, attain the optimal regret similar to the Robbins estimator for both bounded or subexponential priors.  In addition, when $\pi$ has heavier (polynomial) tails, the NPMLE achieves the corresponding optimal regret while Robbins estimator provably fails \cite{shen2022empirical}. 
However, the downside of $g$-modeling is its much higher computational cost. For example, \prettyref{eq:NPMLE} entails solving an infinite-dimensional convex optimization. Although in one dimension faster algorithms akin to Frank-Wolfe have been proposed \cite{L83general,jana2022optimal}, for multiple dimensions existing solvers essentially all boil down to maximizing the weights over a discretized domain \cite{koenker2014convex} which clearly does not scale with the dimension.


\subsection{Empirical Bayes via Empirical Risk Minimization}
\label{sec:ERM}

In this paper we propose a new approach for Poisson empirical Bayes by incorporating a framework based on \textit{empirical risk minimization} (ERM) and the needed technology from learning theory, notably, the \textit{offset Rademacher complexity}, refined via localization, to establish the optimality of the achieved regret. In contrast to $f$-modeling and $g$-modelling that aim at approximating the mixture density and the prior respectively, the main idea is to directly approximate the Bayes rule by solving a suitable ERM subject to certain structural constraints satisfied by the Bayesian oracle. 
We note that a similar technique has been applied earlier in \cite{barbehenn2022nonparametric} to the Gaussian model; however, the theoretical guarantees therein are highly suboptimal.

The benefits of the ERM-based methodology are manifold:
\begin{enumerate}
    \item Unlike the Robbins method, the constrained ERM produces an estimator that enjoys the same regularity as that of the Bayes rule, at a small permillage of the computational cost of $g$-modeling methods such as the NPMLE and other minimum-distance estimators.
    
    \item The ERM-based estimator is scalable to high dimensions and runs in time that is polynomial in both $n$ and the dimension $d$. In contrast, all existing algorithms for NPMLE are essentially grid-based and scales poorly with the dimension as $n^{\Theta(d)}$.

    \item The ERM approach invites powerful tools from 
    empirical processes theory (such as Rademacher complexity and variants) to bear on its regret.

    \item The flexibility of the ERM framework allows one to easily incorporate extra constraints or replace the function class by more powerful ones (such as neural nets) in order to tackle more challenging empirical Bayes problems in high dimensions for which there is no feasible proposal so far. 
\end{enumerate}

To summarize, the ERM can be seen as an alternative solution to the empirical Bayes problem, that excels over the Robbins method in terms of retaining the regularity properties of the Bayes estimator, and is computationally much efficient than the other existing non-parametric alternatives. We will also show that theoretically it achieves the optimal regret for certain light-tailed classes of priors. Whether these guarantees carry over to the heavy-tailed classes of prior, where the Robbins method is known to be suboptimal and NPMLE is known to be optimal \cite{shen2022empirical}, is beyond the scope of the current paper.

Next we describe the construction of the ERM-based empirical Bayes estimator in details. To derive the objective function for the ERM, note that using $\fbest(X)=\EE\qth{\theta|X}$, we have 
\begin{align*}
        \fbest = \argmin_f \EE[(f(X) - \theta)^2]& =\argmin_f \EE[(f(X))^2 - 2\theta f(X)]\\
        &=\argmin_{f}\EE\qth{f(X)^2-2Xf(X-1)},
\end{align*}
where we get the last step applying the identity 
$\EE\qth{\theta f(X)}=\EE\qth{Xf(X-1)}$ for $X\sim\Poi(\theta)$.
Since $\fbest$ is monotone, this naturally leads to the ERM-based estimator
\begin{align}\label{eq:erm-est}
 \ferm \in \argmin_{f \in \calF}\hat\EE [f(X)^2 - 2Xf(X - 1)],
\end{align}
 where 
 $\hat \EE [h(X)]\eqdef\frac 1n \sum_{i=1}^nh(X_i)$ denotes the empirical expectation of a function $h$ based on the sample $X_1,\dots,X_n$, and the minimization \prettyref{eq:erm-est} is over the class of monotone functions $\calF = \{f: f(x)\le f(x + 1), \forall x\ge 0\}$. 
We also note that the solution~\eqref{eq:erm-est} is only uniquely specified on the set $S\triangleq \{X_1, \ldots, X_n\}\cup \{X_1-1,\ldots,X_n-1\}$, which can be easily computed by an algorithm akin to isotonic regression (see 
\prettyref{lmm:erm_construction}).
We then extend this solution to the whole $\bbZ_+$ in a piecewise constant manner: 
for those $x<\min S$, set $\ferm(x) = 0$;
for those $x>\max S = X_{\max} \eqdef \max\{X_1,\ldots,X_n\}$, set $f(x)=f(X_{\max})$;
for the remaining $x\not\in S$, set
$\ferm (x) = \ferm (\max\{y\in S: y\le x\})$.
This natural piecewise constant extension clearly retains monotonicity.

 
 We note that the above construction of the ERM-based 
 empirical Bayes estimator can be done in a principled way for other mixture models than Poisson (see \prettyref{tab:erms}). Indeed, \cite{barbehenn2022nonparametric} was the first to 
apply this approach to the Gaussian mixture model. However,  only the \textit{slow rate} of $\frac{\polylog(n)}{\sqrt{n}}$  is obtained for the regret by applying standard empirical process theory. In addition, they use extra constraints, such as the ones based on bounded derivatives, bounds on the parameter space, etc. These constraints can be used to further improve upon the practical performances of the ERM estimator we use for the Poisson model; however the corresponding analysis is beyond the scope of the current paper. One of the major technical contributions of the present paper is to introduce a suitable version of the \textit{offset Rademacher complexity} \cite{liang2015learning} that leads to the \textit{fast rate}  of
$\frac{\polylog(n)}{n}$  (even with the optimal logarithmic factors!)

\begin{table}[H]
	\centering
	\begin{tabular}{|c|c|c|c|}
		\hline 
		Mixture & $p(X|\theta)$ & Bayes estimator & ERM Objective\\
		\hline 
		Geo$(\theta)$ 
		& $\theta^X (1-\theta)$ 
		& $1 - \frac{p_{\pi}(X + 1)}{p_{\pi}(X)}$ 
		& $\hat{\bbE}[f(X)^2 - 2f(X)  + 2f(X - 1)\indc{X > 0}]$\\
		\hline 
		NB$(r, \theta)$
		& $\binom{k+r-1}{k} (1-\theta)^r \theta^k$
		& $\frac{X+1}{X+r}\frac{p_{\pi}(X + 1)}{p_{\pi}(X)}$
		& $\hat{\bbE}[f(X)^2 - 2 \frac{X+1}{X+r} f(X-1) \indc{X > 0}]$\\
		\hline 
		$\mathcal{N}(\theta, 1)$
		& $ \frac{1}{\sqrt{2\pi}}\exp\pth{-\frac{(X-\theta)^2}{2}}$
		& $X + \frac{p'_{\pi}(X)}{p_{\pi}(X)}$
		& $\hat{\bbE}[f(X)^2 - 2Xf(X) + 2 f'(X)]$\\
		\hline 
		$\text{Exp}(\theta)$ 
		& $\theta\exp(-\theta X)$
		& $-\frac{p'_{\pi}(X)}{p_{\pi}(X)}$
		& $\hat{\bbE}[f(X)^2 - 2 f'(X)]$\\
		\hline 
	\end{tabular}
	\caption{ERM objectives for other mixture models: geometric, negative binomial, normal location, and exponential distributions.}
	\label{tab:erms}
\end{table}


\subsection{Regret optimality}

In addition to its conceptual simplicity and computational advantage, the ERM-based estimator comes with strong statistical guarantees which we now describe. Let $\calP[0,h]$ denote the class of all priors supported on the interval $[0,h]$ and $\subexpo(s)$ the set of all $s$-subexponential distributions on $\reals_+$, namely $\subexpo(s)=\sth{G:G([t,\infty)])\leq 2e^{-t/s},\forall t>0}$. Our main result is as follows:

\begin{theorem}[Regret optimality of ERM-based estimators]
\label{thm:main}
	Let $\ferm$ be defined in \eqref{eq:erm-est}, with $\calF$ the class of all monotone functions on $\integers_+$. Then there exist s a constant $C>0$ such that for any $h,s>0$, 
	$$\sup_{\pi\in \calP([0,h])} \Regret_\pi(\ferm)\leq \frac {C\max\{1, h\}^3}n\pth{\log n\over \log\log n}^2,\quad 
	\sup_{\pi\in {\subexpo(s)}} \Regret_\pi(\ferm)\leq \frac {C\max\{1, s\}^3}n (\log n)^3.$$
\end{theorem}	

The regret bounds in \prettyref{thm:main} match the minimax lower bounds in \cite[Theorem 2]{polyanskiy2021sharp} up to constant factors, thereby establish the strong optimality of the ERM-based empirical Bayes estimators.
Finally, as a side remark, we mention that, one can show that a monotone projection of the Robbins estimator, given by $\hat
		f_{\text{\sf mono-Rob}}=\argmin_{f \in \calF} \hat \EE[(f(X)-\frob(X))^2]$, also attains similar regret guarantees as in \prettyref{thm:main}.
  This is outside the scope of the current paper.

\subsection{Multiple dimensions}
The ERM-based estimator \prettyref{eq:erm-est} can be easily extended to the $d$-dimension Poisson model. For clarity, we use the bold fonts to denote a vector, e.g., $\btheta=\pth{\theta_1,\dots,\theta_d},\vtheta_{i}=(\theta_{i1},\dots,\theta_{id}),\vX=(X_1,\dots,X_d),\vX_{i}=(X_{i1},\dots,X_{id}), \vx=(x_1,\dots,x_d)$, etc. Let $\pi$ be a prior distribution on $\reals_+^d$. Consider the following data-generating process
\begin{align}
	\vtheta_i\iiddistr \pi \qquad X_{ij} \inddistr \Poi(\theta_{ij}).
\end{align}
Note that the marginal distribution of the multidimensional Poisson mixture is given by 
$$p_\pi(\vx)=\int_{\vtheta} \prod_{i=1}^d e^{-\theta_i}{\theta_i^{x_i}\over x_i!} d\pi(\vtheta),\quad \vx\in\integers_+^d.$$
Similar to \prettyref{eq:regret}, let us define the  regret of a given estimator  $\vf:\bbZ_+^d\to \bbR_+^d$ as
\begin{align}
	\Regret_\pi(\vf) =\EE\qth{\|\vf(\vX)-\vtheta\|^2}-\EE\qth{\|\vfbest(\vX)-\vtheta\|^2},
\end{align}
where $\vX\sim p_\pi$ is a test point independent from the training sample $\vX_1,\ldots,\vX_n\iiddistr p_\pi$.
For each $\vf$, let $\vf = (f_1, \cdots, f_d)$ where $f_i:\bbZ_+^d\to \bbR_+$. Denote by 
$\vfbest$ the Bayes estimator, whose $i$-th coordinate $f_i^*$ is given by 
$$
	\fbest_i(\vx)
	=\EE[\theta_{i}| \vx]
	={\int_{\vtheta} \theta_i \prod_{j=1}^d e^{-\theta_i}{\theta_i^{x_i}\over x_i!} d\pi(\vtheta)\over p_{\pi}(\vx)}=(x_{i} + 1)\frac{p_{\pi}(\vx + \ve_i)}{p_{\pi}(\vx)},\quad i=1,\dots,d,$$
where $\ve_i$ denote the $i$-th coordinate vector. 
Using Cauchy-Schwarz, one can show that the Bayes estimator for the $i$-th coordinate is increasing in the $i$-th coordinate of the input if all other coordinates are fixed, i.e.,
\begin{align}
	\fbest_i(\vx)\leq \fbest_i(\vx+\ve_i), \quad\forall i=1,\dots,d, \quad \forall \vx\in\bbZ_+^d
\end{align}
This leads to the following ERM procedure.
\begin{align}\label{eq:est-multidim}
	\hat{\vf}_{\mathsf{erm}}&=\argmin_{\vf\in\calvF}\quad \hat{\EE}\qth{\norm{\vf(\vX)}^2 - 2\sum_{j=1}^d X_{j} f_j(\vX - \ve_i)},\nonumber\\
	\calvF=&	\{\vf: \bbZ_+^d \to \bbR_+^d: f_i(\vx)\le f_i(\vx+\ve_i), \forall i=1, \cdots, d, \forall \vx\in\bbZ_+^d\}.
\end{align}
We again note that $\hat{\vf}_{\mathsf{erm}}$ is not uniquely defined for all $\vx\in \bbZ_+^d$. To specify a minimizer, note that $(\ferm)_j$, the $j$-th coordinate of $\hat{\vf}_{\mathsf{erm}}$, is uniquely defined on $S\triangleq \{\vX_i\}\cup\{\vX_i-\ve_j\}$. 
We may extend it to $\integers_+^d$ in the same manner as the one-dimensional case of \prettyref{eq:erm-est} in a piecewise constant manner. 
That is, for each $\vx\not\in S$, 
if there exists $y \ge 0$ such that $\vx - y\ve_j\in S$, 
we set $(\ferm)_j(\vx) = (\ferm)_j(\min_{y\ge 0\atop \vx - y\ve_j\in S} \vx - y\ve_j)$. 
Otherwise, set $(\ferm)_j(\vx)=0$. 
By convention, we also define $(\ferm)_j(-\ve_j)=0$. 

\begin{theorem}\label{thm:main_multidim}
	The ERM  estimator \prettyref{eq:est-multidim}
 satisfies the following regret bounds whenever $n \ge d$: 
	\begin{enumerate}
		\item If $\pi$ is supported on $[0,h]^d$, then $\Regret_\pi(\vferm)\leq O(\frac dn\max\{c_1, c_2h\}^{d+2}(\frac{\log (n)}{\log \log (n)})^{d+1})$ ; 
		
		\item If all marginals of $\pi$ are $s$-subexponential 
  for some $s>0$, then\\ $\Regret_\pi(\vferm)\leq O(\frac dn{(\max\{c_3, c_4s\}\log (n))^{d+2}})$, 
	\end{enumerate}
    where $c_1, c_2, c_3, c_4 > 0$ are absolute constants. 
\end{theorem}

We conjecture these regret bounds in Theorem \ref{thm:main_multidim} are nearly optimal and factors like $(\log n)^d$ are necessary. Indeed, for the Gaussian model in $d$ dimensions, the minimax squared Hellinger risk for density estimation is shown to be at least 
$O((\log n)^d/n)$ for subgaussian mixing distributions and the minimax regret is typically even larger. 
A rigorous proof of matching lower bound for Theorem \ref{thm:main_multidim} will likely involve extending the regret lower bound based on Bessel kernels in \cite{polyanskiy2021sharp}
to multiple dimensions; this is left for future work.

\begin{figure}[t]
	
	{%
		\subfigure[Latent $\vtheta_i$'s.]{%
			\label{fig:triangle}
			\includegraphics[width=0.3\textwidth]{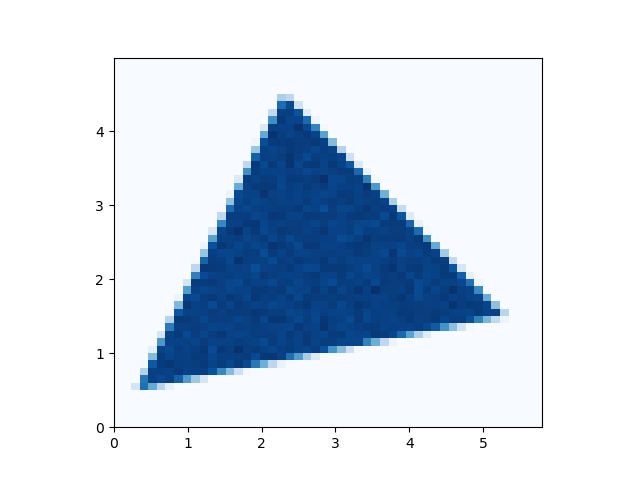}
		}\qquad 
		\subfigure[Observations $\vX_i$'s.]{%
			\label{fig:triangle_obs}
			\includegraphics[width=0.3\textwidth]{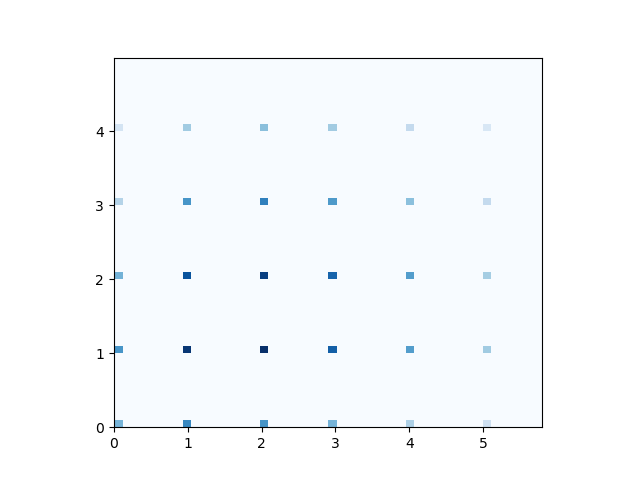}
		}
	    \subfigure[Denoised $\vferm$.]{%
	    	\label{fig:triangle_erm}
	    	\includegraphics[width=0.3\textwidth]{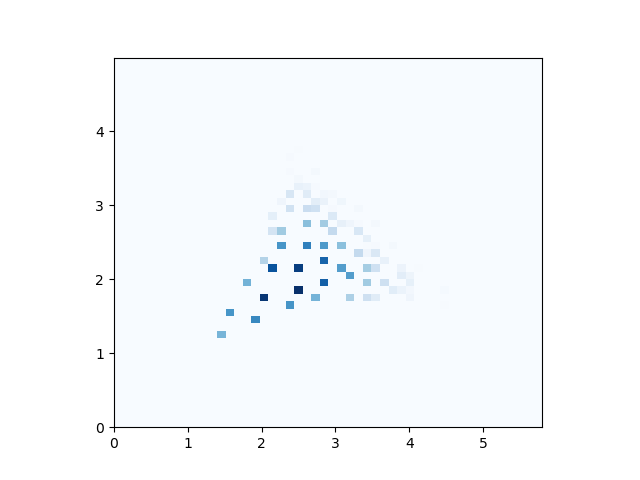}
	    }
	}
{\caption{A two-dimensional experiment with $n=10^6$: Left: $\vtheta_i$'s are sampled uniformly from a triangle. Middle: the observations $\vX_i$'s are drawn independently from $\Poi(\vtheta_i)$, with their empirical distribution shown on the grid $\integers_+^2$ (notice that this is also the MLE estimator for $\vtheta$, hence very different from the empirical Bayes solution). Right: the empirical Bayes denoised version obtained by applying $\ferm$ in \eqref{eq:est-multidim} to $\vX_i$'s.}}
 \label{fig:triangle_plots}
\end{figure}

\begin{remark}[Time complexity]
For the statistical rate of ERM in multiple dimensions to be meaningful, we require $d$ to be significantly smaller than $n$. Nonetheless, even in the dimensions where the regret in \prettyref{thm:main_multidim} is vanishing, the ERM method is computationally much more scalable, compared with the conventional approach based on NPMLE or other minimum-distance estimators. 

To elaborate on this, ERM is a linear program and has a dedicated solver due to its special form. 
NPMLE is an infinite-dimensional convex optimization, and the prevailing solver either discretizes the domain (at least $\sqrt{n}$ level in order to be statistically relevant, thus requires a grid of size $n^{\Theta(d)}$) or runs Frank-Wolfe style iteration, 
which is only known to converge slowly at $\frac{1}{t}$ rate \cite{L83general} and requires mode finding that is expensive in multiple dimensions. 
In contrast, the ERM approach scales much better with the dimension. To evaluate the $d$-dimensional ERM \prettyref{eq:est-multidim}, 
as we will demonstrate in Remark \ref{rmk:ERM-ddim},
if $\ell$ is the number of distinct vector-valued observations $\vX_1,\ldots,\vX_n$, our algorithm runs in $O(d\ell^2)\le O(dn^2)$ time (apart from reading the sample of size $n$). 
An almost linear time $O(d\ell\log \ell)$ algorithm (which is how we implemented in the simulations), exists but is beyond the scope of this paper. (We will describe the basic idea in \prettyref{app:stack}.)


On the empirical side, 
we demonstrate the multidimensional feasibility of ERM by running a simulation with $\vtheta_1,\ldots,\vtheta_n$ sampled uniformly from a triangle with $n=10^6$ and compute the empirical Bayes denoiser $\vferm$ in \prettyref{eq:est-multidim} to $\vX_i\inddistr \Poi(\vtheta_i)$. 
Here, we see that $\vferm$ can recover the triangular structure of the prior, as in \prettyref{fig:triangle_plots}. 
To further compare the computational costs of ERM and minimum distance methods, 
we did a comparison in the statistical software $R$ with the popular package ``REBayes" \cite{koenker2017rebayes} and the results are as follows. With the prior $\mathsf{Unif}(4,30)$ and sample sizes $n=50,500,5000,50000$, we ran both REBayes and ERM 100 times and found that on average the ERM is respectively $21, 50, 212, 588$ times faster. This improvement is even more pronounced ($25, 58, 227.5, 2160$ times) if we supply the empirical distribution to the ERM instead of the full sample.

\end{remark}

\begin{remark}[Comparison with $f$-modelling]
	While both $f$-modelling (i.e.~the Robbins estimator) and the ERM estimator $\ferm$ are asymptotically optimal, 
	we demonstrate more concretely the advantage of $\ferm$ over Robbins. The shortcomings of the Robbins method have been widely observed in practice and  discussed in the existing literature. Most recently, it has been demonstrated in \cite{jana2022optimal} extensively through both simulated and real data experiment. 
	Expanding on \prettyref{fig:triangle}, which compares the performance of the multidimensional Robbins method and $\ferm$ under a uniform prior on the 2d triangle, for $n=10^k, k=4,5,6,7$, we found that the Robbins method achieved a regret of $0.356, 0.0575, 0.00771, 0.00116$ and $\ferm$ achieved a regret of $0.0748, 0.0161, 0.00276, 0.000463$, suggesting a much better performance. 
	On another experiment, 
	we also compared the methods in dimensions $1,2,3,4$ using a product of $\text{Exp}(2)$ distributions as prior, fixing $n=10000$. 
	The Robbins method achieved regrets $0.0125, 0.0607, 0.185, 0.427$; $\ferm$ achieved regrets $0.00422, 0.0208, 0.0660, 0.161$.
\end{remark}

\subsection{Related work}

Empirical Bayes estimation for the Poisson means incorporating shape constraint has a long research thread. However, the majority of the work relies on approximating the Robbins estimator using monotone functions. For example, \cite{maritz1966smooth} used linear approximation to the Robbins estimator and \cite{maritz1969empirical} represented the marginal distribution $p_\pi$ based on a monotone ordinate fit to the Robbins and then used it to compute a maximum likelihood estimation of the ordinates. Both of these papers focus on numerical comparison of the corresponding error guarantees; see \cite[Section 3.4.5]{maritz2018empirical} for a concise exposition. In recent work, \cite{BGR13} discussed the numerical benefits of first performing a Rao-Blackwellization on the Robbins estimator and then using an isotonic regression to impose the monotonicity of the final estimator. An important theoretical contribution to the monotone smoothing of any given empirical Bayes estimator has been proposed in \cite{van1977monotonizing}. Using the monotone likelihood ratio property of the Poisson distribution, it is shown that any estimator (e.g., the Robbins estimator) can be made monotone without increasing the regret. In contrast, our main estimator is computed directly via minimizing an empirical version of the regret. It might be possible to use the monotone smoothing of \cite{van1977monotonizing} to further improve the ERM-estimator which is not pursued in this work.

 As mentioned in Section \ref{sec:ERM}, the application of empirical risk minimization in empirical Bayes has been introduced in the one-dimensional normal mean model by \cite{barbehenn2022nonparametric}. Using the monotonicity of the posterior mean, they construct an empirical Bayes estimator by solving the ERM under monotonicity constraint (see \prettyref{tab:erms}). However, the regret bound they establish is of the slow rate $\polylog(n)\over \sqrt n$ which is highly suboptimal, compared with the nearly optimal rate of $O(\frac{(\log n)^5}{n})$ by \cite{JZ09} (based on the $g$-modeling approach via NPMLE) and $O(\frac{(\log n)^8}{n})$ 
by \cite{li2005convergence} (based on the $f$-modeling approach of polynomial kernel density estimates). As mentioned earlier, the NPMLE is computationally expensive, especially in multiple dimensions due to the reliance on grid-based approximation \cite{koenker2014convex,soloff_multivariate_2021}. 
In contrast, as mentioned before, ERM-based estimators algorithm can be easily constructed for multiple or high dimensions.

The rest of the paper is organized as follows.
\prettyref{sec:offset-rademacher} 
provides a regret upper bound on the ERM-based estimator in one dimension in terms of the offset Rademacher complexities, and a proof sketch for \prettyref{thm:main}. 
 \prettyref{sec:multidim} contains the analysis for the multidimensional ERM-estimator and a proof sketch of \prettyref{thm:main_multidim}. 
Omitted proofs are provided in the appendices.

\section{Regret guarantees for the ERM estimator via Offset Rademacher complexity}\label{sec:offset-rademacher}

\subsection{The ERM algorithm}
As mentioned in the last section, our proposed estimator is based on ERM framework. In many statistical problems, the statistician intends to find a function $f$ that approximates a target statistic $s(X)$ in order to minimize the error $\EE\qth{\ell\pth{s(X),f(X)}}$ for some suitable loss function $\ell$. In the ERM framework, the population average is replaced by the empirical average $\hat \EE\qth{\ell\pth{s(X);f(X)}}$ over the training sample. There is a rich literature on using such methods to approximate nonparametric target functions. See, for example, \cite{nemirovskii1985nonparametric,van1990estimating}  for regression problems, \cite{barron1991complexity,barron1991minimum,barron1994approximation} for penalized empirical risk minimization,  \cite{birge1993rates,lugosi1995nonparametric} for consistency results of general nonparametric ERM-estimators, etc. In this paper, we aim to approximate the nonparametric target function  $\fbest$ (the Bayes rule) by minimizing $\EE\qth{(\fbest(X)-f(X))^2}$. As shown in \prettyref{sec:ERM}, in the Poisson mixture model, this can be equivalently expressed as minimizing $\EE\qth{f(X)^2-2Xf(X-1)}$ and we minimize the corresponding empirical loss over the class of all monotone functions.
Isotonic minimization of such quadratic loss is easy to compute; \cite{best_active_1990} showed that monotone projection can be done in linear time. In the following lemma we present one such minimization algorithm that we use in numerical analyses. 
The proof is deferred to \prettyref{app:technical}. 
\begin{lemma}\label{lmm:erm_construction}
	Let $a_1 < \cdots < a_n$ be a sequence of non-negative integers and $\{v_i\}_{i=1}^n,\{w_i\}_{i=1}^n$ be two non-negative sequences with $v_n>0$ and $\max\{v_i, w_i\}>0$ for all $i$. 
	Consider the iterative $b_i$ 
	\[
	b_i = 
	\begin{cases}
		1 & i = 0\\
		1 + \argmin_{b_{i-1}\le i^*\le n} \frac{\sum_{i=b_{i-1}}^{i^*}w_i}{\sum_{i=b_{i-1}}^{i^*}v_i} & i \geq 1\\
	\end{cases}
	\]
	where the fraction is $+\infty$ whenever the denominator is 0, 
	and where tie exists at $\argmin$, choose biggest such $i^*$. 
	We stop at $b_m = n + 1$. Then the solution to $$\ferm = 
 \argmin_{f\in\calF} 
 \sum_{i=1}^n v_{i}f(a_i)^2 - 2w_{i}f(a_i)$$ is given as
	\[
	\forall i = 1, \cdots, m, \forall x: b_m \le x < b_{m+1}: 
	\ferm(a_x)= \frac{\sum_{i=b_{m}}^{b_{m+1}-1}w_i}{\sum_{i=b_{m}}^{b_{m+1}-1}v_i}.
	\]
\end{lemma}
\begin{remark}
	Making the restriction $v_i\ge 0$ and $v_n>0$ ensures that our solution will be well-formed. 
     To apply this algorithm to estimate $\ferm$, 
     let $\{a_1, \cdots, a_k\}\subseteq \{1, \cdots, X_{\max}\}$ be such that either $N(a_i) > 0$ or $N(a_i + 1) > 0$.
     Here, $v_i = N(a_i)$ and $w_i=(a_i + 1)N(a_i + 1)$. 
     Our choice of $a_i$'s for $i=1, \dots, k$ ensures that $\max\{v_i, w_i\} > 0$, 
     and also $v_k > 0$. 
     
\end{remark}

\begin{remark}
\label{rmk:ERM-ddim}
\prettyref{lmm:erm_construction} can be applied to compute the ERM estimator \eqref{eq:est-multidim} for the multivariate case. Recall that the function class $\calvF$ 
 dictates the following form of monotonicity: 
 for each vector $\vx'=(x_1', \cdots, x'_{j-1}, x'_{j+1}, \cdots, x_d)$ of length $d - 1$, 
	we define 
	\begin{equation}\label{eq:cj_class}
		C_{j}(\vx')\triangleq \{\vx\in\bbR_+^d: x_i=x_i', \forall i\neq j\}
	\end{equation}
	Here are several examples for $d = 3$: 
	\[
	C_0((0, 0))=\{(0,0, 0), (1, 0, 0), (2, 0, 0), \cdots\}
	\quad 
	C_1((0, 0)) = \{(0, 0, 0), (0, 1, 0), (0, 2, 0), \cdots\}
	\]\[
	C_2((0, 0)) = \{(0, 0, 0), (0, 0, 1), (0, 0, 2), \cdots\}
	\]
	Then  $\vf\in\calvF$ if and only if for each $j\in[d]$, $f_j$ restricted on each $C_j(\vx')$ is monotone in the $j$-th coordinate of the argument.
	Since the objective function 
	$\hat{\bbE}[\norm{\vf(\vX)}^2 - 2\sum_{j=1}^d X_jf_j(\vX - \ve_j)]$ is separable, 
	for each $j$ we may determine $(\ferm)_j$ by partitioning the samples $\vX_1, \cdots, \vX_n$ into classes of $C_j(\vx')$, and then apply \prettyref{lmm:erm_construction} to each class. 
\end{remark}

To bound the regret of such ERM-estimators, we used the technique of Rademacher complexities. The Rademacher analysis, popularized by \cite{koltchinskii2001rademacher,mendelson2002rademacher,bartlett2002model}, etc., uses a symmetrization argument to bound the error using the supremum of an empirical process of the form $\sup_{g\in \calF}\frac 1n\sum_{i=1}^n{\epsilon_i g(X_i)}$, where $\epsilon_{1},\dots,\epsilon_n$ are \iid Rademacher random variables, and $\calF$ is some suitable function class. The complexity of such a function class is often characterized by the VC dimension or the covering numbers. An immediate bound on the complexity is produced by the uniform convergence bound when $\calF$ is chosen to be the class of all possible candidate functions, however, this has been shown to guarantee only a slow rate of regret ($\frac{1}{\sqrt{n}}$), which is the case in the prior work \cite{barbehenn2022nonparametric} that applies the ERM approach to the Gaussian model. An improvement on this is made by restricting $\calF$ to be a smaller class, for example using the techniques of local Rademacher complexities \cite{bartlett_local_2005,koltchinskii_rademacher_2004,lugosi2004complexity} which analyzes the complexity within a small ball around the target function, the empirical minimizer, etc. We employ a similar technique of using function classes with smaller complexity. Note that the empirical minimizer in \eqref{eq:erm-est} satisfies the following regularity property.
\begin{lemma}\label{lmm:max_support}
	Let $\ferm$ be the ERM-estimator defined in \eqref{eq:erm-est}. 
 Let $X_{\max} = \max\{X_1,\ldots,X_n\}$.
 Then $\max_{0\leq x\leq X_{\max}}\ferm(x)\leq X_{\max}$.
\end{lemma}
\begin{proof}
	Recall that $\ferm$ is characterized by piecewise constancy, where for each maximal interval $I$ on which $\ferm$ is constant (maximal in the sense we cannot extend $I$ further), 
	we have 
	\[
	\forall x_0\in I: 
	\ferm(x_0) = \frac{\sum_{x\in I} (x + 1) N(x + 1)}{\sum_{x\in I} x N(x)}
	\]
	Now that we have defined $\ferm(x) = \ferm(X_{\max})$ for all $x > X_{\max}$, 
	it suffices to show that $\ferm(X_{\max})\le X_{\max}$. 
	Indeed, there exists an $i^*\le X_{\max}$ such that 
	\begin{align}
		\ferm(k) &= \frac{\sum_{i=i^*}^{X_{\max}} (i+1) N(i+1)}{\sum_{i=i^*}^{X_{\max}} N(i)}
        \nonumber\\
		&\stepa{=}\frac{\sum_{i=i^*+1}^{X_{\max}} i N(i)}{\sum_{i=i^*}^{X_{\max}} N(i)}
		\le  \frac{\sum_{i=i^*+1}^{X_{\max}} {X_{\max}} N(i)}{\sum_{i=i^*}^{X_{\max}} N(i)}
		={X_{\max}}(1- \frac{N(i^*)}{\sum_{i=i^*}^k N(i)})
		\le {X_{\max}}
	\end{align}
	where (a) is due to $N({X_{\max}} + 1)=0$. 
    \end{proof}

When $X_1,\dots,X_n$ are generated from the Poisson mixture with either a  compactly supported or subexponential prior, the above result implies that the value of ERM-estimator is at most $\Theta(\polylog(n))$ with high probability. 
This, in essence, dictates the required complexity of the function class. 

\subsection{Risk bounds for ERM via Rademacher complexities}

\prettyref{lmm:max_support} shows that $\ferm$ coincides with the ERM over the following more restrictive class
\begin{align}
	\calF_* \eqdef \sth{f: f\ \text{is monotone}, f(X_{\max})\leq \max\sth{X_{\max}, \fbest(X_{\max})}}.
 \label{eq:calFstar}
\end{align}
Note that $\calF_*$ is a (random) class that depends on the sample maximum. Furthermore, since it depends on the unknown ground truth $f^*$, it is not meant for data-driven optimization but only for theoretical analysis of the ERM \prettyref{eq:erm-est}.
In addition, our work utilizes the quadratic structure of the empirical loss to obtain a stronger notion of the Rademacher complexity measure, which closely resembles and is motivated by the offset Rademacher complexity introduced in \cite{liang2015learning}. 
\begin{theorem}\label{thm:rademacher-bound}
 Let $\mathcal{F}$ be a convex function class that contains the Bayes estimator $f^*$. 
 Let $X_1,\dots,X_n$ be a training sample drawn iid from $p_\pi$,  $\epsilon_1,\dots,\epsilon_n$ an independent sequence of \iid Rademacher random variables, and $\hat{f}$  the corresponding ERM solution. 
 Then for any function class $\calF_{p_n}$ depending on 
the empirical distribution $p_n = \frac{1}{n} \sum_{i=1}^n \delta_{X_i}$ that includes $\hat f$ and $\fbest$ we have 
	\begin{equation}
		\Regret_\pi(\hat f) \le \frac 3nT_1(n) + \frac 4nT_2(n)
	\end{equation}
    where
    \begin{align}
    T_1(n) & = \EE\qth{\sup_{f\in\calF_{p_n}\cup\calF_{p_n'}} 
    \sum_{i = 1}^n (\epsilon_i - \frac 16)(f(X_i) - {\fbest}(X_i))^2},  \label{eq:t1}  \\    
    T_2(n) & = \EE\Bigg[\sup_{f\in\calF_{p_n}\cup\calF_{p_n'}}
    \sum_{i = 1}^n \Bigg\{2\epsilon_i ({\fbest}(X_i)({\fbest}(X_i) - f(X_i))
    - X_i({\fbest}(X_i - 1) - f(X_i - 1)))
   	\nonumber \\
    & \quad \quad  - \frac 14 ({\fbest}(X_i) - f(X_i))^2\Bigg\}\Bigg], \label{eq:t2}
    \end{align}
    and 
 $\calF_{p_n'}$ is defined in the same way as $\calF_{p_n}$ 
  with respect to an independent copy of $X_1,\dots,X_n$.
\end{theorem}
\begin{proof}
Define 
\begin{align}\label{eq:def-R}
	\sfR(f)= \EE\qth{f(X)^2-2Xf(X-1)},
	\quad \hat \sfR(f)= \hat \EE\qth{f(X)^2-2Xf(X-1)}.
\end{align}
We first note that $\hat{f}$ satisfies the following inequality, thanks to the convexity of $\mathcal{F}$:
\begin{align}\label{eq:geom-ineq}
		\hat{\sfR}(h) - \hat{\sfR}(\hat{f})\ge \hat{\EE}[(h - \hat{f})^2], \quad \forall h\in\calF.
\end{align}
To show this claim, since $\calF$ is convex,for any $\epsilon\in[0, 1]$, $(1-\epsilon)\hat{f}+\epsilon h$ is inside the class $\calF$, 
so with $\hat{\sfR}(\hat{f})\le \hat{\sfR}((1-\epsilon)\hat{f}+\epsilon h)$ we have 
\begin{align*}
	\frac{\partial}{\partial \epsilon}\hat{\sfR}((1-\epsilon)\hat{f}+\epsilon h)
	&=2\hat{\EE}[(h(X) - \hat{f}(X))((1-\epsilon)\hat{f}(X) + \epsilon h(X))
	- X (h(X - 1) - \hat{f}(X - 1))]
\end{align*}
By the ERM minimality of $\hat{f}$, such derivative must be nonnegative when evaluated at 0. 
That is, 
\begin{equation}
	\hat{\EE}[(h(X) - \hat{f}(X))\hat{f}(X)- X (h(X - 1) - \hat{f}(X - 1))]\ge 0
\end{equation}
Therefore, 
evaluating the difference gives us 
\begin{align}
	&\hat{\sfR}(h) - \hat{\sfR}(\hat{f}) - \hat{\EE}[(h(X) - \hat{f}(X))^2]
	\nonumber\\
	&=\hat{\EE}[(h(X)^2 - \hat{f}(X)^2) - 2X(h(X - 1) - f(X - 1))]-\hat{\EE}[(h(X) - \hat{f}(X))^2]
	\nonumber\\
	&=2\hat{\EE}[h(X)\hat{f}(X) - \hat{f}(X)^2 - X(h(X - 1) - \hat{f}(X - 1))]
	\ge 0
\end{align}
as desired. Then using $\Regret_\pi(\hat{f})= \sfR(\hat{f}) - \sfR({\fbest})$ we get
\begin{align}
	& ~\Regret_\pi(\hat{f}) 
        \nonumber\\
	&\le \EE\qth{\sfR(\hat{f}) - \sfR({\fbest}) + \hat{\sfR}({\fbest}) - \hat{\sfR}(\hat{f}) - \hat{\EE}({\fbest} - \hat{f})^2}
	    \nonumber\\
	&= \EE\Big[(\sfR(\hat{f}) - \sfR({\fbest})-\EE[(\fbest-\hat f)^2]) + (\hat{\sfR}({\fbest}) - \hat{\sfR}(\hat{f}) + \hat \EE[(\fbest-\hat f)^2])
	\nonumber\\
	&\quad + \EE[(\fbest-\hat f)^2] - 2\hat{\EE}[({\fbest} - \hat{f})^2]\Big]
	\nonumber\\
        &=\EE\left[\hat\EE[2{\fbest}(X)({\fbest}(X) - \hat{f}(X)) - 2X({\fbest}(X - 1) - \hat{f}(X - 1))]\right. \nonumber\\
	&\quad \left.-\EE[2{\fbest}(X)({\fbest}(X) - \hat{f}(X)) - 2X({\fbest}(X - 1) - \hat{f}(X - 1))] - \frac 14 (\hat\EE[({\fbest} - \hat{f})^2] + {\EE}[({\fbest} - \hat{f})^2])\right]\label{eq:m0_1}
       \\
	&\quad +\EE\left[\frac 54\EE[({\fbest}(X) - \hat{f}(X))^2]
	- \frac 74\hat{\EE}[({\fbest(X)} - \hat{f}(X))^2]\right].\label{eq:m0_2}
\end{align}

We separately bound the two terms \prettyref{eq:m0_1} and \prettyref{eq:m0_2} in the above display in terms of the Rademacher complexities using the following symmetrization result.

\begin{lemma}\label{lmm:symmetrization}
	Let $\epsilon_1, \cdots, \epsilon_n$ as independent  Rademacher symbols. Let $T,U$ be two operators mapping  $f(x)$ to $Tf(x)$ and $Uf(x)$. Then 
 \begin{align*}
		&~\EE\qth{\sup_{f\in\calF_{p_n}}[\EE[Tf(X)] - \hat{\EE}[Tf(X)] - (\EE[Uf(X)] + \hat{\EE}[Uf(X)])]}
		\le \frac{2}{n}\EE\Bigg[
		\sup_{f\in\calF_{p_n}\cup \calF_{p_n'}}
		\sum_{i=1}^n \epsilon_i Tf(X_i) - Uf(X_i)
		\Bigg]
	\end{align*}
	where $p'_{n}$ is an independent copy of the empirical distribution $p_n$. 
\end{lemma}
\begin{proof}
	Here, we note that the symmetrization technique has been introduced in \cite[p.11-12]{liang2015learning}. 
	However, given that we are taking a supremum over a data-dependent subclass of $\mathcal{F}$,  some extra care needs to be taken. 
	
		\begin{align}
			&~\EE[\sup_{f\in\calF_{p_n}\cup \calF_{p_n'}}
			\hat{\EE}'[T(f(X))] - \hat{\EE}[T(f(X))] - (\hat{\EE}'[U(f(X))] + \hat{\EE}[U(f(X))])
			]
			\nonumber\\&
			\stackrel{\text{(a)}}{=} \frac 12\EE[\sup_{f\in\calF_{p_n}\cup \calF_{p_n'}}
			\hat{\EE}'[T(f(X))] - \hat{\EE}[T(f(X))] - (\hat{\EE}'[U(f(X))] + \hat{\EE}[U(f(X))])
			]
			\nonumber\\& 
			+ \frac 12\EE[\sup_{f\in\calF_{p_n}\cup \calF_{p_n'}}
			\hat{\EE}[T(f(X))] - \hat{\EE}'[T(f(X))] - (\hat{\EE}'[U(f(X))] + \hat{\EE}[U(f(X))])
			]
			\nonumber\\&=\frac{1}{2n}
			\EE[\sup_{f, g\in\calF_{p_n}\cup \calF_{p_n'}}
			\sum_{i=1}^n T(f)(X_i') - T(f)(X_i) - U(f)(X_i) - U(f)(X_i')
			\nonumber\\&
			+ T(g)(X_i) - T(g)(X_i') - U(g)(X_i) - U(g)(X_i')]
			\nonumber\\&
			\le \frac{1}{2n}
			\EE[\sup_{f_1, g_1\in\calF_{p_n}\cup \calF_{p_n'}}
			\sum_{i=1}^n T(g_1)(X_i) - T(f_1)(X_i) - U(f_1)(X_i) - U(g_1)(X_i)]
			\nonumber\\&
			+\frac{1}{2n}
			\EE[\sup_{f_2, g_2\in\calF_{p_n}\cup \calF_{p_n'}}
			\sum_{i=1}^n T(f_2)(X_i') - T(g_2)(X_i') - U(f_2)(X_i') - U(g_2)(X_i')]
			\nonumber\\&
			\stackrel{\text{(b)}}{=} \frac{1}{n}
			\EE[\sup_{f, g\in\calF_{p_n}\cup \calF_{p_n'}}
			\sum_{i=1}^n T(g)(X_i) - T(f)(X_i) - U(f)(X_i) - U(g)(X_i)]
			\nonumber\\&
			\stackrel{\text{(c)}}{\le } \frac{2}{n}
			\EE[\sup_{f\in\calF_{p_n}\cup \calF_{p_n'}}
			\sum_{i=1}^n \epsilon_i T(f)(X_i) - U(f)(X_i)]
		\end{align}
		where (a), (b) are symmetry and (c) is Jensen's inequality. 
\end{proof}
As $f\in \calF_{p_n}$, applying the last lemma to previous display above, with the choice for the first expectation \eqref{eq:m0_1} 
\begin{align*}
	Tf(x)=-\qth{2{\fbest(x)}(\fbest(x)-f(x))-2x(\fbest(x-1)-f(x-1))}, \quad
	Uf(x)=\frac 14(\fbest(x)-f(x))^2\,,
\end{align*}
and the choice for the second expectation \eqref{eq:m0_2}
$Tf(x)=\frac 32(\fbest(x)-f(x))^2, 
	Uf(x)=\frac 12(\fbest(x)-f(x))^2$, 
we get the desired result.
\end{proof}

\subsection{Controlling the Rademacher complexities}

To prove \prettyref{thm:main}, we apply \prettyref{thm:rademacher-bound} with the function class $\calF_{p_n}=\calF_*$ defined in \prettyref{eq:calFstar}. Denote by $\calF_{p_n'}=\calF_*'$ its independent copy based on a fresh sample $X_1',\dots,X_n'$. Let us define the following generalization of \eqref{eq:t1} and \eqref{eq:t2}: For $b>1$,
\begin{align}
	T_1(b,n) & = \EE\qth{\sup_{f\in\calF_*\cup\calF_*'} 
		\sum_{i = 1}^n (\epsilon_i - \frac 1b)(f(X_i) - {\fbest}(X_i))^2},  \label{eq:t1-new}  \\    
	T_2(b,n) & = \EE\Bigg[\sup_{f\in\calF_*\cup\calF_*'}
	\sum_{i = 1}^n 2\epsilon_i ({\fbest}(X_i)({\fbest}(X_i) - f(X_i)) \nonumber \\
	& \quad \quad - X_i({\fbest}(X_i - 1) - f(X_i - 1))) - \frac 1b ({\fbest}(X_i) - f(X_i))^2\Bigg]. \label{eq:t2-new}
\end{align}
Then we have the following bound on the complexities.

\begin{lemma}\label{lmm:T1T2-bounds}
	Let $\pi\in \calP[0,h]$ with $h$ being either a constant or $h=s\log n$ for some $s>0$. Let $M:=M(n,h)>h$ be such that \begin{itemize}
		\item $\sup_{\pi\in \calP([0,h])}\PP_{X\sim p_\pi}\qth{X>M}\leq \frac 1{n^7}$.
		\item For $X_i\simiid p_\pi,$ $\EE\qth{X_{\max}^k}\leq c(k) M^k$ for $k=1,\dots,4$ and absolute constant $c>0$.
	\end{itemize}
	Then there exists a constant $c_0(b)>0$ such that
	\begin{equation}
		T_1(b,n),T_2(b, n)\leq c_0(b)\pth{\max\{1,h^2\}M+ \max\{1,h\}M^2}\label{eq:lmm_t2bound}.
	\end{equation}
\end{lemma}

The first condition on the probability is an artifact of the proof. In general, any tail bounds on the random variable $X$ that decay polynomially in $n$, such as the ones satisfied by bounded priors or priors with subexponential tails, are good enough for our proofs to go through.

\begin{proof}[Proof of \prettyref{lmm:T1T2-bounds}]
We consider the following notations. 
\begin{equation}
	N(x) = \sum_{i=1}^n \indc{X_i = x}
	\qquad 
	\epsilon(x) = \sum_{i=1}^n \epsilon_i\indc{X_i = x}
\end{equation}
where $\epsilon_1, \cdots, \epsilon_n$ are independent Rademacher symbols. 

\paragraph{Bound on $T_2(b,n)$:} Using $f(-1)=0$ we note that
\begin{align}
	&\sum_{i = 1}^n 2\epsilon_i ({\fbest}(X_i)({\fbest}(X_i) - f(X_i)) - X_i({\fbest}(X_i - 1) - f(X_i - 1)))
	- \frac 1b ({\fbest}(X_i) - f(X_i))^2
	\nonumber \\
	=& \sum_{x\ge 0} 2\epsilon(x) ({\fbest}(x)({\fbest}(x) - f(x)) - x({\fbest}(x - 1) - f(x - 1)))
	- \frac {N(x)}{b} ({\fbest}(x) - f(x))^2
	\nonumber \\
	=& \sum_{x\ge 0} 2(\epsilon(x) {\fbest}(x) - (x+1) \epsilon(x + 1))({\fbest}(x) - f(x))
	- \frac {N(x)}{b} ({\fbest}(x) - f(x))^2
\end{align}
In view of the above, we can bound $T_2(b,n)$ using the sum of the following two terms
\[t_1(n) \triangleq
\EE\{\sup_{f\in{\calF_*\cup\calF_*'}}
[\sum_{x\ge 0} 2(\epsilon(x) {\fbest}(x) - (x+1) \epsilon(x + 1))({\fbest}(x) - f(x))
- \frac {N(x)}{b} ({\fbest}(x) - f(x))^2] \indc{N(x) > 0}\}
\]
\[t_0(n) \triangleq 
\EE\{\sup_{f\in{\calF_*\cup\calF_*'}}
[\sum_{x\ge 0} -2(x+1) \epsilon(x + 1)({\fbest}(x) - f(x))] \indc{N(x)=0}\}.
\]
For analyzing the term $t_1(n)$, since $N(x) > 0$, 
using $2ax - bx^2 \le \frac{a^2}{b}$ for any $a, x$ and $b > 0$ we get
\begin{equation}\label{eq:i1n_bound}
	t_1(n)\le 
	b\cdot \EE\qth{\sum_{x\ge 0} \frac{(\epsilon(x){\fbest}(x) - (x+1)\epsilon(x+1))^2}{N(x)} \indc{N(x) > 0}}
\end{equation}
Using $\EE\sth{\epsilon(x)|X_1,\dots,X_n}=0$ and $\EE\qth{(\epsilon(x))^2|X_1,\dots,X_n}=N(x)$ we get
\[
\EE\qth{\frac{({\fbest}(x)\epsilon(x) - (x+1)\epsilon(x+1))^2}{N(x)}\indc{N(x)>0}}
=\EE\qth{\pth{(\fbest(x))^2 + \frac{(x+1)^2 N(x+1)}{N(x)}}\indc{N(x)>0}}.
\]
Using the results that
\begin{enumerate}[label=(P\arabic*)]
	\item \label{prop:P1} $N(x)\sim \Binom(n,p_\pi(x))$ and  for absolute constant $c'>0$ \cite[Lemma 16]{polyanskiy2021sharp} $$\EE\qth{\indc{N(x)>0}\over N(x)}\leq c'\min\sth{np_\pi(x),\frac 1{np_\pi(x)}},$$
	\item conditioned on $N(x), N(x+1)\sim\Binom(n-N(x), \frac{p_{\pi}(x+1)}{1 - p_{\pi}(x)})$,
	\item $\fbest(x)=(x+1){p_\pi(x+1)\over p_\pi(x)}=\EE\qth{\theta|X=x}\leq h$ for all $x\geq 0$,
	\item Since for every $x > 0$, ${x^ye^{-x}\over y!}\leq {y^ye^{-y}\over y!}\leq \frac 1{\sqrt{2\pi y}}$ (Stirling's), we have \begin{align}\label{eq:poi(x)>=1}
		p_\pi(y)<\frac 1{\sqrt{2\pi y}},\quad y\geq 1,
	\end{align}
\end{enumerate}
we continue  \eqref{eq:i1n_bound} to get
\begin{align*}
	\frac 1b t_1(n)
	&\leq  \EE\qth{\sum_{x\geq 0}\fbest(x)^2\indc{N(x)>0}}+\sum_{x\geq 0}(x+1)^2{np_\pi(x+1)\over 1-p_{\pi}(x)}\EE\qth{\indc{N(x)>0}\over N(x)}
	\nonumber\\
	&\leq  h^2\EE\qth{1 + X_{\max}}+{np_{\pi}(1)\over 1-p_{\pi}(0)}\EE\qth{\indc{N(0)>0}\over N(0)}
	+n\sum_{x\geq 1}(x+1)^2{p_\pi(x+1)}\EE\qth{\indc{N(x)>0}\over N(x)}
	\nonumber\\
	&\leq  h^2\EE\qth{1 + X_{\max}}+{c'p_{\pi}(1)\over (1-p_{\pi}(0))p_\pi(0)}
	+c'h\sum_{x\geq 1}(x+1)\min\sth{(np_\pi(x))^2,1}.
\end{align*}
Let $M>h$ be as in the lemma statement. For the second term notice that ${p_{\pi}(1)\over (1-p_{\pi}(0))p_\pi(0)}\leq \max\sth{1,h}$. For the third term, we use the bound
\begin{align}
	&h\sum_{x\geq 1}(x+1)\min\sth{(np_\pi(x))^2,1}
	\leq hM^2+ h\sum_{x\geq M}(x+1)\min\sth{(np_\pi(x))^2,1} 
	\nonumber\\
	&\leq hM^2+ 2n^2h\sum_{x\geq M}x(p_\pi(x))^2
	\stepa{\leq} hM^2+ 2n^2h^2 \PP_{X\sim p_\pi}[X>M]\leq 2(hM^2+{2h^2\over n^5}).
\end{align}
where (a) is due to that $xp_{\pi}(x) = f^*(x - 1)p_{\pi}(x - 1)\le h$ for all $x\ge 1$. 
We finally note that since $h$ is either constant or in the form $O(s\log n)$ for some constant $s$, the term $\frac{h^2}{n^5}$ can be neglected. 

Next, we evaluate $t_0(n)$. As $|\epsilon(x+1)|\leq N(x+1)$ and $N(x+1)=0$ for $x\geq X_{\max}$ we get
\begin{align}\label{eq:m4}
	t_0(n) &
	\le\EE\qth{\sum_{x\geq 0} 
		2(x+1)N(x +1)\sup_{f\in {\calF_*\cup\calF_*'}} \abs{\fbest(x)-f(x)}\indc{N(x)=0}}
	\nonumber\\
	&\le\EE\qth{\sum_{x= 0}^{X_{\max}-1} 
		2(x+1)\pth{\fbest(x)+X_{\max}+X_{\max}'}N(x +1)\indc{N(x)=0}}.
\end{align}
Let $M>0$ be as in the lemma statement and $A=\sth{X_{\max}\leq M,X_{\max}'\leq M}$. Then $\PP\qth{A^c}\leq \frac 2{n^6}$ via the union bound argument. 
Thus we have, for some absolute constant $c > 0$, 
\begin{align}\label{eq:m3}
	&~\EE\qth{\sum_{x= 0}^{X_{\max}-1} 
		2(x+1)\pth{\fbest(x)+X_{\max}+X_{\max}'}N(x +1)\indc{N(x) = 0}\cdot \indc{A^c}}
	\nonumber\\
	\leq &~\bbE\qth{
		X_{\max}(h + X_{\max} + X_{\max}')\sum_{x=0}^{X_{\max} - 1}N(x + 1)\indc{N(x)=0}\indc{A^c}
	}
    \nonumber\\
	\stepa{\leq} & ~n\EE\qth{\pth{X_{\max}}\cdot \pth{h+X_{\max}+X_{\max}'}\indc{A^c}}
	\stepb{\leq} n\sqrt{\EE\qth{\pth{h+X_{\max}+X_{\max}'}^4}}\sqrt{\PP\qth{A^c}}
	\leq {cM^2\over n^2}. 
\end{align}
with (a) due to that $\sum_{x=0}^{X_{\max} - 1}N(x + 1)\le \sum_{x=0}^{\infty}N(x)=n$, 
and (b) the Cauchy-Schwarz inequality and $\EE\qth{X_{\max}^4}\lesssim 2M^4$. 

For each $x\le M$, define $q_{\pi, M}(x) \triangleq \frac{p_{\pi}(x)}{\bbP_{X\sim p_{\pi}}[X\le M]}$. Note that $\PP\qth{N(x)=0|A}=(1-q_{\pi, M}(x))^n$ and conditioned on the set $A$ and $\{N(x)=0\}$, the random variable $N(x+1)$ has $\Binom\pth{n,{q_{\pi, M}(x+1)\over 1-q_{\pi, M}(x)}}$ distribution. This implies  
\begin{align*}
	&~\EE\qth{\left.\sum_{x= 0}^{X_{\max}-1} 
		2(x+1)\pth{\fbest(x)+X_{\max}+X_{\max}'}N(x +1)\indc{N(x)=0}\right|A}
	\nonumber\\
	&\leq \sum_{x=0}^{M-1} 2(x+1)(h+2M)\EE\qth{N(x+1)|N(x)=0,A}\PP\qth{N(x)=0|A}
	\nonumber\\
	&\leq \sum_{x=0}^{M-1} 2(x+1)(h+2M){nq_{\pi, M}(x+1)\over 1-q_{\pi,M}(x)}\pth{1-q_{\pi, M}(x)}^n
	\nonumber\\
	&= \sum_{x=0}^{M-1} 2(h+2M)\fbest(x)nq_{\pi, M}(x)\pth{1-q_{\pi, M}(x)}^{n-1}
	\stepa{\leq} 2Mh(h+2M).
\end{align*}
where (a) uses $\fbest(x)\le h$ for all $x$, and also 
$nw\pth{1-w}^{n-1}\le (1-\frac{1}{n})^{n-1} < 1$ for all $w\in [0, 1]$. 
We conclude our proof by combining the above with \eqref{eq:m3}.

\paragraph{Bound on $T_1(b,n)$:} Denote $m_b=b+1$. Conditional on the sample $X_1,\dots,X_n$ and $X_1,\dots,X_n$, given any $f\in \calF_*\cup\calF_*'$ define 
$$v(f)=\min\sth{\min\sth{x:f(x)\leq m_bh},X_{\max}}.$$
Then using the above definition we get for each $f\in \calF_*\cup\calF_*'$, conditional on the samples,
\begin{align}
	&\sum_{i=1}^n (\epsilon_i - \frac 1b)(f(X_i) - {\fbest}(X_i))^2
	=
	\sum_{x:N(x) > 0}(\epsilon(x) - \frac 1b N(x))(f(x) - {\fbest}(x))^2
	\nonumber\\
	&=\pth{\sum_{x=0}^{v(f)} + \sum_{x =v(f)+1}^{X_{\max}}}(\epsilon(x) - \frac 1b N(x))(f(x) - {\fbest}(x))^2
	\nonumber\\
	&\le m_b^2h^2\sum_{x=0}^{X_{\max}} \max\sth{\epsilon(x) - \frac 1b N(x),0}\label{eq:m11}
	\\&\quad + \sup_{v\geq 0}\sth{\sup_{m_bh \leq f \le X_{\max}}\sth{\sum_{x > v}^{X_{\max}} (\epsilon(x) - \frac 1b N(x))(f(x) - {\fbest}(x))^2}}.\label{eq:m12}
\end{align}
For the first term \prettyref{eq:m11}, 
we invoke the following lemma, to be proven in \prettyref{app:binom_tail_proof}. 
\begin{lemma}\label{lmm:binom_tail_rad}
	For each $x$ and $b>1$, conditioned on $X_1^n$ we have 
	\[\bbE[\max\{\epsilon(x) - \frac 1bN(x), 0\}]
	\le \frac{1 - \frac 1b}{e\cdot D(\frac{1+\frac 1b}{2} || \frac 12)}
	\]
\end{lemma}
For brevity, we denote $N_b\triangleq \frac{1 - \frac 1b}{e\cdot D(\frac{1+\frac 1b}{2} || \frac 12)}$. 
This gives us 
\begin{equation}\label{eq:m11_bound}
	\EE\qth{m_b^2h^2\left.\sum_{x=0}^{X_{\max}}
		\max\sth{\epsilon(x) - \frac 1b N(x), 0}\right|X_1^n}
	\leq N_bm_b^2h^2 \bbE[(1 + X_{\max})].
\end{equation}
For the second term \eqref{eq:m12}, 
we note that for any $f$ with values in $[m_bh,X_{\max}]$, we have ${m_b-1\over m_b}f\leq f-\fbest\leq f$ and hence
\begin{align}
	&(\epsilon(x) - \frac 1b N(x))(f(x) - {\fbest}(x))^2
	\le \max\left\{\pth{\epsilon(x) - \frac 1b N(x)}, \pth{\frac{m_b-1}{m_b}}^2\pth{\epsilon(x) - \frac 1b N(x)}\right\}f(x)^2.
\end{align}
Now given that $-N(x)\le \epsilon(x)\le N(x)$, define function $g:[-1, 1]\to\bbR$ given by 
\begin{equation}
	g(x) = \max\left(\pth{x - \frac 1b}, \left(\frac{m_b-1}{m_b}\right)^2\pth{x-\frac 1b}\right)
\end{equation}
Since $g$ is the maximum of two linear functions, it is convex, 
and therefore bounded by the line joining their endpoints, 
$(-1, -(\frac 1b+1)\cdot \left(\frac{m_b-1}{m_b}\right)^2)$ and $(1, 1-\frac 1b)$. 
Now define: 
\begin{equation}
	\alpha =\frac 12\left[\pth{1 + \frac 1b}\cdot \left(\frac{m_b-1}{m_b}\right)^2 + \pth{1 - \frac 1b} \right];
	\quad 
	\beta = \frac 12\left[\pth{1 + \frac 1b}\cdot \left(\frac{m_b-1}{m_b}\right)^2 - \pth{1 - \frac 1b} \right]
	=\frac{1}{2b(b+ 1)}
\end{equation}
using the fact that $m_b = b+1$. 
Note that $0<\beta<\alpha$. Then we have $g(x)\le \alpha x - \beta$ for all $x\in [-1, 1]$. 
Hence, we have 
\begin{equation}
	(\epsilon(x) - \frac 1b N(x))(f(x) - f^*(x))^2
	\le (\alpha\epsilon(x) - \beta N(x))f(x)^2
\end{equation}
Hence \eqref{eq:m12} can be bounded by, modulo a constant multiplicative factor $c_2(b)$ depending on $b$,
\begin{align}\label{eq:m2}
	&\sup_{v\geq 0}\sth{\sup_{m_bh \leq f \le X_{\max}}\sth{\sum_{x > v}^{X_{\max}} (\epsilon(x) - \frac 1b N(x))(f(x) - {\fbest}(x))^2}}
	\nonumber\\
	&\leq c_2(b)\qth{\sup_{v\geq 0}\sth{\sup_{m_bh \leq f \le X_{\max}}\sth{\sum_{x > v}^{X_{\max}} (\epsilon(x) - \frac {\beta}{\alpha} N(x))f(x)^2}}}.
\end{align}
Note that the above $f$-based maximization problem is a linear programming of the form
$$\sup_{a_1,\dots,a_k}\sum_{i=1}^kv_ia_i,\quad
(m_bh)^2\leq a_1\dots\leq a_k\leq \pth{X_{\max}}^2,$$
with $k=X_{\max}+1$. The optimization happens on the corner points of the above convex set, that are given by $X_{\max}+1$ length vectors of the form 
$$\sth{(m_bh)^2,\dots,(m_bh)^2,\pth{X_{\max}}^2,\dots,\pth{X_{\max}}^2}.$$
This implies we can bound \eqref{eq:m2} by 

\begin{align}\label{eq:m6}
	(m_bh)^2\sum_{x=0}^{X_{\max}} \max\sth{\epsilon(x)-{\beta\over \alpha}N(x),0}+ \pth{X_{\max}}^2\sup_{v\geq 0}\sth{\sum_{x>v}^{X_{\max}} (\epsilon(x)-{\beta\over \alpha}N(x))}.
\end{align}
The bound of the first term, conditional on the data, is given as per \prettyref{lmm:binom_tail_rad} as 
$m_b^2h^2N_b(1 + X_{\max})$. For the second term, 
we first note the following result.

\begin{lemma}\label{lmm:binom_penal}
	Let $c > 0$ be given. 
	For $\epsilon=(\epsilon_1, \cdots, \epsilon_n)$ $n$ independent Rademacher symbols, denote 
	\begin{equation}
		L_c(\epsilon) = \max_{0\le j\le n} \left\{\sum_{i=1}^j \epsilon_i - cj\right\}
	\end{equation}
	Then $\EE[L_c(\epsilon)]\le M_c$ 
	where $M_c \triangleq 1 + (1 - \exp(-D(\frac{c + 1}{2} || \frac 12)))^{-2}$. 
\end{lemma}
The proof of the above result is provided in \prettyref{app:technical}. 

Therefore, using \prettyref{lmm:binom_penal}, we have 
\[
\EE\qth{\left.\sup_{v\geq 0}\sth{\sum_{x>v}^{X_{\max}} (\epsilon(x)-{\beta\over \alpha}N(x))}\right| X_1^n}
\le \EE[\sup_{w: 0\le w\le n}(\epsilon_{w+1} + \cdots + \epsilon_n)-{\beta\over \alpha}(n - w)]
\leq c(b)
\]
for some constant $c(b)>0$ via
\prettyref{lmm:binom_penal}. 
Thus we get
\begin{equation}\label{eq:m7}
	\EE\qth{\pth{X_{\max}}^2\sup_{v\geq 0}\sth{\sum_{x>v}^{X_{\max}} (\epsilon(x)-{\beta\over \alpha}N(x))}\Bigg|X_1,\dots,X_n}
	\le c(b)(1 + X_{\max})^2. 
\end{equation}
Combining \eqref{eq:m2}, \eqref{eq:m6}, and \eqref{eq:m7} we get
\begin{align}\label{eq:m8}
	\EE\qth{\sup_{v\geq 0}\sth{\sup_{m_bh \leq f \le X_{\max}}\sth{\sum_{x > v}^{X_{\max}} (\epsilon(x) - \frac 1b N(x))(f(x) - {\fbest}(x))^2}}\Bigg|X_1^n}
	\leq c_3(b)\pth{h^2(1+X_{\max}) + (1+X_{\max})^2}
\end{align}
for a constant $c_3(b)$ depending on $b$. Then taking expectation on both the sides and using the definition of $M$ in the lemma statement we finish the proof.

\end{proof}

\subsection{Proof of Regret optimality (\prettyref{thm:main})}

We use the above result to first prove the
regret bound for bounded priors in 
$\calP([0,h])$. Note that by \prettyref{lmm:bounded_prior} and \prettyref{lmm:exp_decay_lemma}, 
there are constants $c_1, c_2>0$ such that for any fixed $h>0$ such that $M=\max\{c_2, c_1h\}\cdot{\log n\over \log\log n}$ satisfies both conditions in \prettyref{lmm:T1T2-bounds}, and we get $O(\frac{\max\{1, h^3\}}{n}({\log n \over \log \log n})^2)$ bound on the regret, which is optimal up to constants that possibly depend on $h$.

Next we extend the above proof to the subexponential case. Given $\pi\in\mathsf{SubE}(s)$ define the truncated version $\pi_{c,n}[\theta\in\cdot ] = \pi[\theta\in\cdot \mid \theta\le c\log n]$ for $c>0$. Then we have the following reduction.
\begin{lemma} \label{lmm:prior-truncate}
	There exists constants $c_1,c_2, c_3>0$ such that
    $$\mathsf{Regret}_{\pi}(\ferm)\le \mathsf{Regret}_{\pi_{c_1s,n}}(\ferm) + \frac {\max\{c_2, c_3s\}}n.$$
\end{lemma}
\begin{proof}
Let $\pi\in\mathsf{SubE}(s)$, then there exists a constant $c(s)\triangleq 11s$ by the definition of $\subexpo(s)$ such that 
\begin{equation}
	\varepsilon = \bbP[\theta > c(s)\log n] \le \frac{1}{n^{10}},\quad \theta\sim \pi
\end{equation}
Denote, also, the event $E = \{\theta_i\le c(s)\log n, \forall i=1, \cdots, n\}$; 
we have $\bbP[E^c]\le n^{-9}$. 
Let $\pi_{c(s), n}$ as the truncated prior $\pi_{c(s), n}[\theta\in\cdot ] = \pi[\theta\in\cdot \mid \theta\le c(s)\log n]$. 
Define $\text{mmse}(\pi) \triangleq \min_{f}\bbE_{\theta\sim\pi}[(f(X) - \theta)^2]$ 
(i.e. the error by the Bayes estimator). 
Then we may use \cite[Equation 131]{polyanskiy2021sharp} to obtain 
\begin{align}\label{eq:1-dim-reduction}
	\mathsf{Regret}_{\pi}(\ferm)
	\le \mathsf{Regret}_{\pi_{c, n}}(\ferm)
	+\text{mmse}(\pi_{c, n}) - \text{mmse}(\pi) + \bbE_{\pi} [(\ferm(X) - \theta)^2\indc{E^c}]
\end{align}
By \cite[Lemma 2]{wu_functional_2012}, 
$\text{mmse}(\pi_{c, n}) - \text{mmse}(\pi)\le \frac{\varepsilon}{1-\varepsilon}\text{mmse}(\pi)
\le 2\varepsilon$ whenever $\varepsilon\le \frac 12$. 
In addition, \prettyref{lmm:max_support} entails that $\ferm(X)\le X_{\max}$, 
which means that $\bbE[\ferm^4(X)]\le \bbE[X_{\max}^4]\le O(\max\{1, s^4\}(\log n)^4)$ as per \prettyref{lmm:xmax_bound}. 
Meanwhile, for all $\pi\in\mathsf{SubE}(s)$ we have 
$\bbE_{\pi}[\theta^4]\in O(s^4)$. 
This means $\bbE_{\pi}[(\ferm - \theta)^4]\lesssim_s (\log n)^4$. 
Thus by Cauchy-Schwarz inequality 
\begin{equation*}
	\bbE_{\pi} [(\ferm(X) - \theta)^2\indc{E^c}]
	\le \sqrt{\bbP[E^c] \bbE_{\pi}[(\ferm(X) - \theta)^4]}
	\le \sqrt{n^{-9}\bbE_{\pi}[(\ferm(X) - \theta)^4]}
	\lesssim \frac {\max\{1, s^2\}}{n}.
\end{equation*}
\end{proof}

Given this lemma, it suffices to bound $\mathsf{Regret}_{\pi_{c, n}}(\ferm)$. Then by \prettyref{lmm:subexp_prior} and \prettyref{lmm:exp_decay_lemma} there exist constants $c_1, c_2>0$ such that $M=\max\{c_1, c_2s\}\log n$ satisfies both the requirements in \prettyref{lmm:T1T2-bounds}. Hence we get the desired regret bound of $O(\frac{\max\{1, s^3\}(\log n)^3}{n})$.

\section{Regret bounds in multiple dimensions}\label{sec:multidim}

To prove the regret bound for the multidimensional estimator $\hat{\vf}=(\hat f_1,\dots, \hat f_d)$ we use the approximation error for the different coordinates. In particular, similar to \eqref{eq:def-R} we define
\begin{align}\label{eq:r_multidim}
	\sfR(\vf)\triangleq \EE\qth{\norm{\vf(\vX)}^2 - 2\sum_{i=1}^d X_{i} f_i(\vX - \ve_i)},
	\quad 
	\hat{\sfR}(\vf)
	\triangleq \hat{\EE}\qth{\norm{\vf(\vX)}^2 - 2\sum_{i=1}^d  X_{i} f_i(\vX - \ve_i)}
\end{align}
and note that
\begin{align}\label{eq:regret-R}
	\Regret_\pi({\vferm})=\EE\qth{\sfR(\vferm)-\sfR( \vf^*)}
\end{align}
As mentioned before, in the multidimensional setup our estimator is produced by optimizing over the class of coordinate-wise monotone functions $\calvF$ in \eqref{eq:est-multidim} and $\vf^*\in\calvF$ as well.
Using the quadratic structure of the regret and the convexity of $\calvF$, we can mimic the proof of \eqref{eq:geom-ineq} to get
\begin{align}
	\hat \sfR(\vf)-\hat \sfR(\hat \vf)\geq \hat\EE\qth{\|\vf-\hat\vf\|^2},\quad \vf\in \calvF.
\end{align}
Then following a similar argument as in \eqref{eq:m0_1}, \eqref{eq:m0_2}, using \eqref{eq:regret-R} we have 
\begin{align}
	\Regret_\pi({\vferm})
	&\le \bbE\qth{\sfR(\hat\vf) - \sfR(\vf^*) + \hat{\sfR}(\vf^*) - \hat{\sfR}(\hat\vf) - \hat{\EE}\norm{\vf^* - \hat\vf}^2 }
	\nonumber\\
	&= \EE\Biggl[\hat\EE\bigg[\sum_{j=1}^d 2{\fbest_j}(\vX)({\fbest_j}(\vX) - \hat{f_j}(\vX)) - 2\vX_j({\fbest_j}(\vX - \ve_j) - \hat{\vf_j}(\vX - \ve_j))\bigg]\nonumber\\
	&\quad -\EE\bigg[\sum_{j=1}^d2{\fbest_j}(\vX)({\fbest_j}(\vX) - \hat{\vf_j}(\vX)) - 2X_j({\vfbest_j}(\vX-\ve_j) - \hat{\vf_j}(\vX - \ve_j))\bigg]\nonumber\\
	&\quad -\frac 14 (\hat\EE\qth{\|{\vfbest} - \hat{\vf}\|^2} + {\EE}[\|{\vfbest}(\vX) - \hat{\vf}(\vX)\|^2])\Biggr]
	\label{eq:m5a}\\
	&\quad+\EE\left[\frac 54\EE[\|{\vfbest}(\vX) - \hat{\vf}(\vX)\|^2] - \frac 74\hat{\EE}[\|{\vfbest(\vX)} - \hat{\vf}(\vX)\|^2]\right]\,. \label{eq:m5b} 
\end{align}
As \prettyref{lmm:symmetrization} is still directly applicable in the multidimensional setting, applying it with
\begin{equation*}
	T(\vf(\vx)) = -\sum_{j=1}^d[2\vf^*_j(\vx)(f^*_j(\vx) - f_j(\vx)) - 2x_j(f^*_j(\vx-\ve_j) - f_j(\vx-\ve_j))],
	\quad 
	U(\vf(\vx)) = \frac 14 \norm{\vf^*(x) - \vf(x)}^2
\end{equation*}
to bound \prettyref{eq:m5a} and 
with $T(\vf(\vx)) = \frac 32 \norm{\vf^*(x) - \vf(x)}^2,\ U(\vf(\vx)) = \frac 12 \norm{\vf^*(x) - \vf(x)}^2$ to bound \prettyref{eq:m5b} we get: for any function class $\calvF_{p_n}$ depending on 
the empirical distribution $p_n$ of the sample $\vX_1,\dots,\vX_n$ that includes $\vferm$ and $\vfbest$ and its independent copy $\calvF_{p_n'}$ based on an independent sample $\vX_1',\dots, \vX_n$
\begin{align}\label{eq:m1}
	\Regret_\pi(\vferm)&\leq ~\frac 3n\EE\qth{\sup_{\vf\in{\calvF_{p_n}\cup\calvF_{p_n'}}} 
		\sum_{i = 1}^n (\epsilon_i - \frac 16)(f_j(\vX_i) - {\fbest_j}(\vX_i))^2 
	}\nonumber\\
	&\quad + \frac 2n\EE\Biggl[\sup_{\vf\in{\calvF_{p_n}\cup\calvF_{p_n'}}}
	\sum_{i = 1}^n 2\epsilon_i (\fbest_j(\vX_i) (\fbest_j(\vX_i) - f_j(\vX_i)) - X_{ij}(\fbest_j(\vX_i-\ve_j)
	\nonumber\\
	&\qquad \quad- f_j(\vX_i - \ve_j)))- \frac 14 ({\fbest_j}(\vX_i) - f_j(\vX_i))^2\Biggr]
\end{align}
To achieve the best possible bound we choose $\calvF_{p_n}$ with low complexity. 
Note that the objective function $\sfR$ defined in \prettyref{eq:r_multidim} is separable into sum of individual loss functions. 
Thus, given the definition of $\calvF$ in \prettyref{eq:est-multidim}, 
for each coordinate $j$ and each class $C_j(\vx')$ defined in \prettyref{eq:cj_class}, we have 
\[
(\ferm)_j|_{C_j(\vx')} = \argmin_{f\in\mathcal{F}_1}\hat{\bbE}\qth{f_j(\vX) - 2X_jf_j(\vX - \ve_j) | \vX\in C_j(\vx')},\qquad \forall \vx'\in \bbR_+^{d-1}.
\]
where $\mathcal{F}_1$ is the class of all one-dimensional monotone function from $\bbZ_+\to\bbR_+$. 
Considering this for all classes $C_j(\vx')$ and from \prettyref{lmm:max_support}, we have 
\begin{align}\label{eq:xjmax}
	(\ferm)_j(\vX_i) \le X_{j, \max},\ X_{j,\max}\eqdef\max_{i=1}^n X_{ij},\quad j=1,\dots,d.
\end{align}
Given the sample $\vX_1,\dots,\vX_n$ define the sample based function class
\begin{equation}\label{eq:fpn_multidim}
	\calvF_{*}\eqdef\sth{\vf\in\calvF: f_j(\vX_i)\le \max\sth{\fbest_j(\vX_i), X_{j, \max}},\quad j=1,\dots,d,i=1,\dots,n}.
\end{equation}
Let $\calvF_*'$ be an independent copy of $\calvF_*$. Then simplifying \eqref{eq:m1} with $\calvF_{p_n}=\calvF_*, \calvF_{p_n}=\calvF_*'$ we get
\begin{align}\label{eq:u_multidim}
	\Regret_\pi(\vferm)&\leq \frac 1n\sum_{j=1}^d\pth{3U_1(j,n)+4U_2(j,n)}
	\nonumber\\
	U_1(j, n)&\triangleq \EE\qth{\sup_{\vf\in{\calvF_*\cup\calvF_*'}} 
		\sum_{i = 1}^n (\epsilon_i - \frac 16)(f_j(\vX_i) - {\fbest_j}(\vX_i))^2 
	}\nonumber\\
	U_2(j, n)&\triangleq\EE\left[\sup_{\vf\in{\calvF_*\cup\calvF_*'}}
	\sum_{i = 1}^n \epsilon_i (\fbest_j(\vX_i) (\fbest_j(\vX_i) - f_j(\vX_i)) - X_{ij}(\fbest_j(\vX_i-\ve_j) \right.
	\nonumber\\
	&\quad \quad \left.- f_j(\vX_i - \ve_j)))- \frac 18 ({\fbest_j}(\vX_i) - f_j(\vX_i))^2\right].
\end{align}

We bound these $2d$ Rademacher complexities to arrive at the results. 
Note that as we want to analyze the supremum over all possible prior distributions whose marginals are subject to the same  tail assumption 
(either supported on $[0,h]$ or $s$-subexponential), 
by the inherent symmetry on the $d$ coordinates, it suffices to consider only a single coordinate, say, the $j$-th, when bounding the offset Rademacher complexity. 
The final regret bound then includes an extra factor of $d$ over this single instance of Rademacher complexity. Note that in our problem the function class $\calvF_{*}$
is supported over the hypercube $\prod_{j=1}^d[0,X_{j,\max}]$. The high-level idea for our analysis is that the effective size of this hypercube, corresponding to different classes of priors, controls the Rademacher complexity and hence the regret upper bound.

\subsection{Bounding Rademacher Complexity for Bounded Prior}

Here we first prove a bound for the generalization of the Rademacher complexities in \eqref{eq:u_multidim} for $b>1$:
\begin{align}\label{eq:u-general_multidim}
	U_1(b,j, n)&\triangleq \EE\qth{\sup_{\vf\in{\calvF_{*}\cup\calvF_{*}'}} 
		\sum_{i = 1}^n (\epsilon_i - \frac 1b)(f_j(\vX_i) - {\fbest_j}(\vX_i))^2 
	}\nonumber\\
	U_2(b, j, n)&\triangleq\EE\qth{\sup_{\vf\in{\calvF_{*}\cup\calvF_{*}'}}
		\sum_{i = 1}^n 2\epsilon_i (\fbest_j(\vX_i) (\fbest_j(\vX_i) - f_j(\vX_i)) - X_{ij}(\fbest_j(\vX_i-\ve_j) \right.
		\nonumber\\
		&\quad \quad \left.- f_j(\vX_i - \ve_j)))- \frac 1b ({\fbest_j}(\vX_i) - f_j(\vX_i))^2}
\end{align}
We have the following result similar to \prettyref{lmm:T1T2-bounds}. 
%
%

\begin{lemma}\label{lmm:U1U2-bounds}
	Let $\pi\in \calP[0,h]$ with $h$ being either a constant or $h=s\log n$ for some $s>0$. Given $\vX_1,\dots,\vX_n$ be \iid observations from $p_\pi$, let $M:=M(n,h)>h$ be such that
	\begin{itemize}
		\item For each coordinate $j=1, \cdots, d$, we have the $j$-th coordinate $X_j$ of $\vX$ satisfying 
		\[\sup_{\pi\in \calP([0,h])^d}\PP_{\vX\sim p_\pi}\qth{X_j > M}\leq \frac {1}{n^7}.\]
		\item For $\beta=1,2,3,4$, constants $c_1(\beta)$ depending on $\beta$ and absolute constant $c>0$
        $$\EE\qth{(X_{j,\max})^4}\leq c M^4,\quad \EE\qth{(1+X_{j,\max})^\beta\prod_{k=1\atop k\neq j}^d(1 + X_{k, \max})}\leq c_1(\beta) M^{d-1+\beta}.$$
	\end{itemize}
	Then there exists a constant $r(b)>0$ such that for all $n\ge d$, 
	\begin{equation}
		U_1(b, j, n), U_2(b, j, n) \le r(b)\sth{\max\{1, h^2\}+ \max\{1, h\}M} (1+M)^d\label{eq:Ubound}.
	\end{equation}
\end{lemma}
\begin{proof}
	At a high level, using the monotonicity of $\mathcal{F}$, for a target coordinate $j$ 
	we partition the samples $\vX_1, \cdots, \vX_n$ such that samples in the same class differ by (possibly) only the $j$-th coordinate. 
	Then for each class, using monotonicity, we mimic the proof for the one-dimensional case. Before proceeding with the proof we define the following notations for all $j=1,\dots,d$ and $ x'\in \integers_+^{d-1}$
	\begin{equation}
		\begin{gathered}
			C_j(\vx') \triangleq 
		\{\vx\in\integers_+^d: x_i=x_i'\ \forall i\leq j-1 \text{ and } x_i=x_{i-1}' \ \forall i\geq j+1\},\\
		 N_j(\vx')=\sum_{\vx\in \integers_+^d} N(\vx)\indc{\vx\in C_j(\vx')}.
		\end{gathered}
	\end{equation}
In addition, we will use multiple times that by union bound we have 
\[
\sup_{\pi\in \calP([0,h])^d}\PP_{\vX\sim p_\pi}\qth{\vX\not\in [0, M]^d}
\le \sum_{i=1}^d \sup_{\pi\in \calP([0,h])^d}\PP_{\vX\sim p_\pi}\qth{X_j > M}
\le \frac{d}{n^7}
\]

\paragraph{Bound on $U_1(b, j, n)$.}
	Denote $m_b=1+b$ and note that for each $\vf\in\calvF$, and for each class $C_j(\vx')$, as $f_j$ is monotone over the $j$-th coordinate of all $\vx$-s in $C_j(\vx')$, there exists 
	$v\triangleq v(f_j, \vx')$ such that for all $\vx\in C_j(\vx')$, 
	$f_j(\vx)\le m_b h$ if and only if $x_{j} \le v$. 
	Using the above we can write
	\begin{align}
		&\sup_{\vf\in\calvF_*\cup \calvF_*'} 
		\sum_{i=1}^n \left(\epsilon_i - \frac 1b\right)(f^*_j(\vX_i) - f_j(\vX_i))^2
		=\sup_{\vf\in\calvF_*\cup \calvF_*'}
		\sum_{\vx:N(\vx) > 0}(\epsilon(\vx) - \frac 1b N(\vx))(f_j(\vx) - f^*_j(\vx))^2
		\nonumber\\
		&=\sup_{\vf\in\calvF_*\cup \calvF_*'}
		\sum_{\vx':N_j(\vx') > 0}\sum_{\vx\in C_j(\vx')}(\epsilon(\vx) - \frac 1b N(\vx))(f_j(\vx) - f^*_j(\vx))^2
		\nonumber\\
		&=\sup_{\vf\in\calvF_*\cup \calvF_*'}
		\sum_{\vx':N_j(\vx') > 0}\left(\sum_{\vx\in C_j(\vx'),x_{j}\le v}+\sum_{\vx\in C_j(\vx'),x_{j} > v}\right) (\epsilon(\vx) - \frac 1b N(\vx)) 
		(f_j(\vx) - f^*_j(\vx))^2
		\nonumber\\
		&\le\sup_{\vf\in\calvF_*\cup \calvF_*'}
		\sum_{\vx':N_j(\vx') > 0}\Biggl(m_b^2h^2\sum_{\vx\in C_j(\vx'),\atop x_{j}\le v} \max\{0, \epsilon(\vx) - \frac 1bN(\vx)\}
		+\sum_{\vx\in C_j(\vx'),\atop x_{j} > v}(\epsilon(\vx) - \frac 1b N(\vx))
		(f_j(\vx) - f^*_j(\vx))^2\Biggr)
		\nonumber\\ 
		&\le 
		m_b^2h^2\sum_{N(\vx) > 0}\max\{0, \epsilon(\vx) - \frac 1bN(\vx)\}
		\nonumber\\ 
		&+ \sth{\sum_{\vx':N_j(\vx') > 0} \sup_{f\in \calvF_*\cup \calvF_*',\atop N_c h\leq f_j\leq X_{j, \max}} \sth{ \sup_{v(\vx')\geq 0}
			\sum_{\vx\in C_j(\vx'),\atop x_{j}> v(\vx')} (\epsilon(\vx) - \frac 1bN(\vx))(f_j(\vx) - f^*_j(\vx))^2}}
		\label{eq:m5}
	\end{align}
    
	As there are at most $\prod_{j=1}^d (1 + X_{j, \max})$ vectors $\vx$ with $N(\vx) > 0$, 
	we apply \prettyref{lmm:binom_tail_rad} to bound the expectation of the first term in the above display as
	\begin{align}\label{eq:m9_multidim}
		&~ m_b^2h^2\bbE[\sum_{N(\vx) > 0}\max\{0, \epsilon(\vx) - \frac 1bN(\vx)\}|\vX_1,\dots,\vX_n]\nonumber \\
        &\leq m_b^2h^2\bbE[ \sum_{N(\vx) > 0} 1|\vX_1,\dots,\vX_n]
		\stepa{\le} r_1(b) m_b^2h^2\prod_{j=1}^d (1 + X_{j, \max}).
	\end{align}
    where (a) followed from \prettyref{lmm:binom_tail_rad} with $r_1(b)={1-1/b\over e\cdot D({1+1/b\over 2}\|\frac 12)}$. 
    
	For the second term in \eqref{eq:m5}, note that for the vectors in the set $C_j(\vx')$, the only coordinate that takes different values is the $j$-th coordinate, and the function $f_j$ is monotone when we condition on the coordinates $\{1,\dots,j-1,j+1,\dots,d\}$.
	It follows that conditional on $\vX_1,\dots,\vX_n$, for this class $C_j(\vx')$, we can mimic the proof for \eqref{eq:m8} in one dimensional case of $T_1(b,n)$ to bound the innermost term as 
	\begin{align*}
		&~\EE\left[\sup_v\sup_{\vf\in\calvF_{*}\cup \calvF_{*}'}
		\left\{\sum_{\vx\in C_j(\vx'),\atop x_{j}> v} \max\{0, (\epsilon(\vx) - \frac 1bN(\vx))(f_j(\vx) - f^*_j(\vx))^2\}\right\}\Bigg|\vX_1^n\right]
		\nonumber\\
		\le & r_2(b)\pth{h^2(1+X_{j, \max}) + (1+X_{j, \max})^2}
	\end{align*}
	for a constant $c(b)$ depending on $b$. Finally, the number of such classes with $N_j(\vx') > 0$ is bounded above by $\prod_{k=1\atop k\neq j}^d (1  + X_{k, \max})$. 
	Therefore, summing over all classes and taking the expectation, and including \prettyref{eq:m9_multidim}, we get the bound 
	\begin{align}\label{eq:multidim_u1bound}
		&U_1(b, j, n)
		= \EE\qth{\sup_{\vf\in\calvF_*\cup \calvF_*'} 
		\sum_{i=1}^n \left(\epsilon_i - \frac 1b\right)(f^*_j(\vX_i) - f_j(\vX_i))^2}
        \nonumber\\
        &\leq r_1(b)m_b^2h^2\EE\qth{\prod_{j=1}^d (1 + X_{j, \max})}
		+r_2(b)\bbE\qth{\prod_{k=1\atop k\neq j}^d (1  + X_{k, \max})
		\cdot (h^2 X_{j, \max}
		+X_{j, \max}^2 )}
		\nonumber\\
        &\leq (r_1(b)+r_2(b))(c_1(1)h^2 + c_1(2)M)M^d.
	\end{align}


	
		
		
		
		

\paragraph{Bounding $U_2(b, j, n)$.}
As per the one dimensional case, we bound the Rademacher complexity term $U_2(b, j, n)$ with $t_0(n) + t_1(n)$, where 
\begin{align}
	t_1(n) \triangleq \;&
	\bbE\left[\sup_{\vf\in\calvF_{*}\cup\calvF_{*}'}
	\sum_{\vx} (2(\epsilon(\vx) f^*_j(\vx) - (x_{j} + 1) \epsilon(\vx + \ve_j))(f^*_j(\vx) - f_j(\vx))\right.
	\nonumber\\
	& \left.  \quad - \frac {N(\vx)}{b} (f^*_j(\vx) - f_j(\vx))^2 \indc{N(\vx) > 0}\right]\label{eq:t1_multidim}
	\\
	t_0(n) \triangleq \;&
	\bbE\left[\sup_{\vf\in\calvF_{*}\cup\calvF_{*}'}
	\sum_{\vx} -2(x_{j}+1) \epsilon(\vx + \ve_j)(f^*_j(\vx) - f_j(\vx)) \indc{N(\vx) = 0}\right]
	\label{eq:t0_multidim}
\end{align}
We first analyze $t_1(n)$. Using the inequality $2ax-bx^2\le \frac{a^2}{b}$  for any $b>0$ we have
\begin{align}
	\frac 1b t_1(n)
	&\le \bbE\qth{\sum_{\vx}
	\frac{(\epsilon(\vx) f^*_j(\vx) - (x_{j} + 1) \epsilon(\vx + \ve_j))^2}{N(\vx)}\indc{N(\vx) > 0}
}
\end{align}
Using the facts
\begin{itemize}
	\item $\bbE[\epsilon(\vx)| \vX_1, \dots, \vX_n] = 0,\ \bbE[\epsilon(\vx)\epsilon(\vx+\ve_j)| \vX_1, \dots, \vX_n] = 0 $
	\item $\bbE[\epsilon(\vx)^2| \vX_1, \dots, \vX_n] = N(\vx)$, and,
	\item $\bbE[N(\vx + \ve_j) \mid N(\vx)] = \frac{(n - N(\vx))p_{\pi}(\vx + \ve_j)}{1 - p_{\pi}(\vx)}\le \frac{np_{\pi}(\vx + \ve_j)}{1 - p_{\pi}(\vx)}$
\end{itemize} 
we  continue the last display to get
\begin{align}\label{eq:m2_multidim}
	\frac 1b t_1(n)
	& \leq \bbE[\sum_{\vx}
	\left(f^*_j(\vx)^2 +\frac{ (x_{j} + 1)^2 N(\vx + \ve_j)}{N(\vx)}\right)\indc{N(\vx) > 0}]
	\nonumber\\ 
	& \leq \bbE[\sum_{\vx} h^2\indc{N(\vx) > 0}]
	+\bbE[\sum_{\vx}\frac{ (x_{j} + 1)^2 np_{\pi}(\vx + \ve_j)}{1 - p_{\pi}(\vx)}\cdot \frac{\indc{N(\vx) > 0}}{N(\vx)}]
	\nonumber\\ 
	& \stepa{\leq} \bbE[\sum_{\vx} h^2\indc{N(\vx) > 0}]
	+c'\cdot 
	\sum_{\vx}\frac{ (x_{j} + 1)^2 np_{\pi}(\vx + \ve_j)}{1 - p_{\pi}(\vx)}\cdot \min\{np_{\pi}(\vx), \frac{1}{np_{\pi}(\vx)}\}
	\nonumber\\ 
	&\stepb{=} \bbE[\sum_{\vx} h^2\indc{N(\vx) > 0}]
	+c'\cdot\sum_{\vx}\frac{ (x_{j} + 1)f_j^*(\vx)}{1 - p_{\pi}(\vx)}\cdot \min\{1, (np_{\pi}(\vx))^2\}
	\nonumber\\ 
	&\stepc{\leq} h^2\bbE[\prod_{j=1}^d(1 + X_{j, \max})]
	+\frac{c'f_j^*(\bm{0})}{1 - p_{\pi}(\bm{0})}
	+c'\sum_{\vx\neq \bm{0}}(x_{j} + 1)f_j^*(\vx)\cdot \min\{1, (np_{\pi}(\vx))^2\}
\end{align}
(here $c'$ is an absolute constant), 
where:
\begin{itemize}
	\item (a) is due to Property \ref{prop:P1} in the analysis of $T_2(b,n)$; 
	
	\item (b) is using 
	$f^*_j(\vx) = (x_{j} + 1)\frac{p_{\pi}(\vx + \ve_j)}{p_{\pi}(\vx)}=\EE\qth{\theta_j|X=\vx}\leq h$; 
	
	\item (c): for the first term, we use the fact that the number of vectors $\vx$ with $N(\vx) > 0$ is bounded by 
	$\prod_{j=1}^d (1 + X_{j, \max})$; for the third term, for each $\vx\neq\bm{0}$ we may choose a coordinate $k$ with $x_k > 0$. 
	Thus setting $p_{\pi_k}$ as the marginal distribution of $x_k$ we have by Stirling's inequality, again, 
	\[
	p_{\pi}(\vx)\le p_{\pi_k}(x_k) \le \sup_{\theta\ge 0} \bbP_{X\sim \Poi(\theta)}[X=x_k] 
	= \sup_{\theta\ge 0}\frac{\theta^{x_k}e^{-\theta}}{x_k!} = \frac{x_k^{x_k}e^{-x_k}}{x_k!}
	\le \frac{1}{\sqrt{2\pi x_k}}\le \frac{1}{\sqrt{2\pi }}
	\]
	and therefore $\frac{1}{1 - p_{\pi}(\vx)}\le \frac{1}{1 - \frac{1}{\sqrt{2\pi }}} \le O(1)$. 
\end{itemize}
Now, the first term in \eqref{eq:m2_multidim} is bounded by
$h^2c(1)M^d$. 
For the second term, using $p_{\pi}(\ve_j)\le 1 - p_{\pi}(\bm{0})$ we have 
$\frac{f^*_j(\bm{0})}{1 - p_{\pi}(\bm{0})}\le \frac{f^*_j(\bm{0})}{p_{\pi}(\ve_j)} = \frac{1}{p_{\pi}(\bm{0})}$, so 
\begin{align}\label{eq:erm1}
	\frac{f^*_j(\bm{0})}{1 - p_{\pi}(\bm{0})}
	\leq \min\sth{\frac{f^*_j(\bm{0})}{1 - p_{\pi}(\bm{0})},\frac{1}{p_{\pi}(\bm{0})}}
	\le 2\max\{f^*_j(\bm{0}), 1\}
	\le 2\max\{h, 1\}
\end{align}
given that $\vf^*$ is bounded by $h$ in each coordinates. 
Finally, the third term in \eqref{eq:m2_multidim} has the following bound:

\begin{align}\label{eq:m1_multidim_2}
	&\sum_{\vx\neq \bm{0}}(x_{j} + 1) f^*_j(\vx)\cdot \min\{1, (np_{\pi}(\vx))^2\}
	\nonumber\\
	& {\le} ~h\sum_{\vx\in [0, M]^d}(x_{j} + 1) + n^2h\sum_{\vx\not\in [0, M]^d}(x_{j} + 1) \cdot p_{\pi}(\vx)^2
	\nonumber\\
	& \stepa{\le} h(1 + M)^{d + 1} + n^2h\PP_{\vX\sim p_{\pi}}\qth{\vX\notin [0,M]^{d}}\bbE_{\vX\sim p_{\pi}}[X_j + 1]
	\stepb{\leq} h(1 + M)^{d + 1} + hdn^{-4}(1+c_1(4)^{1/4}M)
\end{align}
where (a) followed as there are $(1 + M)^{d}$ elements in $[0, M]^d$, 
and (b) is due to the assumptions in \prettyref{lmm:U1U2-bounds} and $\EE[X_{j,\max}+1]\leq \sth{\EE[(X_{j,\max}+1)^4]}^{1/4}$. 
Thus, summarizing \eqref{eq:m2_multidim},\eqref{eq:erm1},\eqref{eq:m1_multidim_2}, we have
\begin{align*}
t_1(n)&\leq c''\cdot b\pth{h^2c_1(1)M^d + \max\{h, 1\} + h(1+M)^{d+1}+hdn^{-4}M}
\nonumber\\
&\leq 2 c'' b\pth{\max\{1, h\}(1 + M)^{d + 1} + \max\{1, h^2\}c_1(1)(1 + M)^{d} + hdn^{-4}M}
\end{align*}
for an absolute constant $c''$, as desired. 
Since $d\le n$, $hdn^{-4}M\le hn^{-3}M < h(1+M)^d$, and can therefore be negleected. 

Next we analyze $t_0(n)$. Since we have $|\epsilon(\vx + \ve_j)|\le N(\vx + \ve_j)$ and $N(\vx + \ve_j) = 0$ for all $\vx $ with 
$\vx + \ve_j \not\in \prod_{k=1}^d [0, X_{k, \max}]$, 
we get
\begin{align}\label{eq:m4_multidim}
	t_0(n) &=
	\EE\qth{\sup_{\vf\in {\calvF_*\cup\calvF_*'}}\sum_{\vx} 
		[-2(x_{j}+1)\epsilon(\vx +\ve_j)(\fbest_j(\vx)-f_j(\vx))\indc{N(\vx)=0}]}
	\nonumber\\
	&\le\EE\qth{\sum_{\vx + \ve_j \in \prod_{k=1}^d [0, X_{k, \max}]}
		2(x_{j}+1)N(\vx +\ve_j)\sup_{f\in {\calF_*\cup\calF_*'}} \abs{\fbest_j(\vx)-f_j(\vx)}\indc{N(\vx)=0}}
	\nonumber\\
	&\le\EE\qth{\sum_{\vx + \ve_j \in \prod_{k=1}^d [0, X_{k, \max}]}
		2(x_{j}+1)\pth{\fbest_j(\vx)+X_{j, \max}+ X_{j, \max}'}N(\vx +\ve_j)\indc{N(\vx)=0}}
\end{align}
where $X_{j, \max}'$ is the maximum of $j$-th coordinate on $n$ samples independent of $\vX_1, \cdots, \vX_n$. 

Define
$A=\sth{\vX_{i}, \vX_{i'}\in [0, M]^d, \forall i=1, \cdots, n}$. 
We have $\PP\qth{A^c}\leq \frac {2d}{n^6}$ via union bound.
Then we have for an absolute constant $c'_1>0$
\begin{align}\label{eq:m3_multidim}
	&~\EE\qth{\sum_{\vx + \ve_j \in \prod_{k=1}^d [0, X_{k, \max}]} 
		2(x_{j}+1)\pth{\fbest_j(\vx)+X_{j, \max}+ X_{j, \max}'}N(\vx +\ve_j)\indc{N(\vx)=0}\cdot \indc{A^c}}
	\nonumber\\
	\leq &~\EE\qth{2(X_{j, \max}+1)\pth{h+X_{j, \max}+ X_{j, \max}'}\sum_{\vx + \ve_j \in \prod_{k=1}^d [0, X_{k, \max}]} 
		N(\vx +\ve_j)\indc{N(\vx)=0}\cdot \indc{A^c}}
	\nonumber\\
	\stepa{\leq}  &~ n\EE\qth{(X_{j, \max} + 1) \pth{h+X_{j, \max}+ X_{j, \max}'}\indc{A^c}}
	\nonumber\\
	\stepb{\leq} &~n\sqrt{\EE\qth{\pth{h+X_{j, \max}+X_{j, \max}'}^2 (X_{j, \max} + 1)^2 }}\sqrt{\PP\qth{A^c}}
	\le ~ c'_1 {hd^{1/2}M^{2}\over n^2} \stepc{\le} \frac{c_1'hM^2}{n},
\end{align}
where (a) is using $\sum_{\vx + \ve_j \in \prod_{k=1}^d [0, X_{k, \max}]} 
N(\vx +\ve_j)\indc{N(\vx)=0}\le \sum_{\vx }N(\vx) = n$, 
(b) is via Cauchy-Schwarz inequality and $\bbE[(X_{j, \max})^4],\bbE[(X'_{j, \max})^4]\leq c M^4$, 
and (c) is because $d\le n$ by our assumption. 

Next, we condition on the event $A$.
Similar to the proof of bound on $T_2(b,n)$ in the one-dimensional setup, we define
$q_{\pi, M}(\vx) \triangleq \frac{p_{\pi}(\vx)}{\bbP_{\vX\sim p_{\pi}}[\vX\in [0, M]^d]}$. 
We have $\PP\qth{N(\vx)=0|A}=(1-q_{\pi, M}(\vx))^n$, and conditioned on the set $A$ and $\{N(\vx)=0\}$,  $N(\vx+\ve_j) \sim \Binom\pth{n,{q_{\pi, M}(\vx+\ve_j)\over 1-q_{\pi, M}(\vx)}}$. Therefore: 
\begin{align*}
	&~\EE\qth{\left.\sum_{\vx + \ve_j \in \prod_{k=1}^d [0, X_{k, \max}]} 
		2(x_{j}+1)\pth{\fbest_j(\vx)+X_{j, \max}+X_{j, \max}'}N(\vx +\ve_j)\indc{N(\vx)=0} \right|A}
	\nonumber\\
	&\leq \sum_{\vx  \in \prod_{k=1}^d [0, M]^d} 2(x_{j}+1)(h+2M)\EE\qth{N(\vx+\ve_j)|\{N(\vx)=0\},A}\PP\qth{N(\vx)=0|A}
	\nonumber\\
	&\leq \sum_{\vx \in \prod_{k=1}^d [0, M]^d} 2(x_{j}+1)(h+2M){nq_{\pi, M}(\vx+\ve_j)\over 1-q_{\pi,M}(\vx)}\pth{1-q_{\pi, M}(\vx)}^n
	\nonumber\\
	&\stepa{=} \sum_{\vx  \in \prod_{k=1}^d [0, M]^d} 2(h+2M)\fbest_j(\vx)nq_{\pi, M}(\vx)\pth{1-q_{\pi, M}(\vx)}^{n-1}
	\leq 2(M + 1)^dh(h+2M).
\end{align*}
where (a) followed using $\fbest_j(\vx)=(x_j+1){p_\pi(\vx+\ve_j)\over p_\pi(\vx)}$ and the definition of $q_{\pi,M}(x+\ve_j)$, and for the last inequality, we used the fact that 
$nx(1-x)^{n-1}\le (1 - \frac{1}{n})^{n-1} < 1$ for all $x$ with $0 < x < 1$ and $\fbest_j(\vx)\leq h$. 
Collecting terms and using $M > h$, we therefore have 
\begin{equation}
	t_0(n)\leq c'_1 {hd^{1/2}M^{2}\over n^2} + h(M + 1)^{d + 1} \leq  c'_2 h(M + 1)^{d + 1} 
\end{equation}
for absolute constants $c'_1,c'_2$ as required.
\end{proof}
\subsection{Proof of Regret bound in the multidimensional setup (\prettyref{thm:main_multidim})}
\label{app:sub-expo-multidim} 


We start by describing the bounds on $\bbE[\prod_{j=1}^d(1 + X_{j, \max})^{k_j}]$ in this multidimensional setting, 
which we claim the following. 
\begin{lemma}\label{lmm:xmax-multidim}
	Given any $s, h > 0$ and integer $\beta\ge 0$ there exist constants $c(\beta), c_1, c_2, c_3, c_4>0$ such that
	\begin{enumerate}
		\item For all $\pi\in\mathcal{P}([0, h]^d)$, 
		$\EE\qth{(1+X_{j,\max})^\beta\prod_{k=1\atop k\neq j}^d(1 + X_{k, \max})}\leq c(\beta)\pth{\max\{c_1, c_2h\}{\log(n)\over \log\log (n)}}^{d-1+\beta}$; 
		
		\item For all $\pi\in\mathcal{P}([0, s\log n]^d)$, 
		$\EE\qth{(1+X_{j,\max})^\beta\prod_{k=1\atop k\neq j}^d(1 + X_{k, \max})}\le c(\beta)(\max\{c_3, c_4s\}\log (n))^{d-1+\beta}$. 
	\end{enumerate}
\end{lemma}

We will defer the proof to \prettyref{app:poimix-props}. 

For $\pi\in\calP([0, h])^d$, by \prettyref{lmm:xmax-multidim}, there exist constants $c_1, c_2$ such that 
we may take $M = \max\{c_1, c_2h\}\frac{\log (n)}{\log \log (n)}$ into \prettyref{lmm:U1U2-bounds}. 
Note that 
This gives the overall regret bound as \\$\frac dn\max\{c_1, c_2h\}^{d+2}(\frac{\log (n)}{\log \log (n)})^{d+1}$.

Now assume that each marginals of $\pi_j$ are of $\mathsf{SubE}(s)$ for some $s > 0$. 
We now show that the multidimensional version of \prettyref{lmm:prior-truncate} applies here. 

Here, we choose $c=c(s)\triangleq 11s$ such that for each $j=1, \cdots, d$, 
we have $\bbP[X_{j} > c(s)\log(n)] \le \frac{1}{n^{10}}$. 
This means that we now have 
\begin{equation}
	\varepsilon = \bbP[\vX\not\in [0, c(s)\log (n)]^{d}] \le
	\sum_{j=1}^d \bbP[X_j > c(s)\log n]
	 \le \frac{d}{n^{10}}
\end{equation}
the middle inequality via union bound on each coordinate. 

Define the event $E = \{\vX_i\in [0, c(s)\log (n)]^d, \forall i=1, \cdots, n\}$, 
and we have $\bbP[E^c]\le dn^{-9}$. 
Again we define the truncated prior 
$\pi_{c,n}[\vX\in \cdot] = \pi[\vX\in\cdot \mid \vX\in [0, c(s)\log (n)]^{d}]$. 
Then, similar to \eqref{eq:1-dim-reduction} in the one-dimensional case, the following equation applies: 
\begin{equation}\label{eq:subexp_truncate}
	\mathsf{Regret}_{\pi}(\vferm)
	\le \mathsf{Regret}_{\pi_{c, n}}(\vferm)
	+\text{mmse}(\pi_{c, n}) - \text{mmse}(\pi) + \bbE_{\pi, c} [\|\vferm(\vX) - \vtheta\|^2\indc{E^c}]
\end{equation}
Given that $\hat{f}_j(\cdot) \le X_{j, \max}$,  we have $\bbE[(\hat{f}_j)^4]\leq \bbE[X_{j, \max}^4]\le  O(s^4(\log n)^4)$ 
by \prettyref{lmm:subexp_prior},
and $\bbE_{\pi}[\theta_{j}^4]\leq O(s^4\log^4 n)$ from the properties of subexponential priors. 
The logic $\bbE_{\pi}[(\fbest_j - \theta_{j})^4]\leq O((s\log n)^4)$ and
\[\bbE_{\pi} [(f_{{\sf erm},j}(\vX) - \theta_{j})^2\indc{E^c}]
\le \sqrt{\bbP[E^c]\bbE_{\pi} [(f_{{\sf erm},j}(\vX) - \theta_{j})^4]}
\lesssim \frac{s^2d^{1/2}}{n^{2}}, \qquad \forall j=1, 2, \cdots, d\]
then follows from there. 
This gives 
$\bbE_{\pi, c} [\|\vferm(\vX) - \vtheta\|^2\indc{E^c}]\le \frac{d^{3/2}}{n^{4}}$ by considering all the $d$ coordinates. 

The identity $\text{mmse}(\pi_c) - \text{mmse}(\pi)\le \frac{\varepsilon}{1-\varepsilon}\text{mmse}(\pi)
\le 2d\varepsilon\le \frac{2d^2}{n^2}$ still applies here in the following sense. 
Let $\vfbest$ be the Bayes estimator corresponding to $\pi$. 
Then denoting $M \triangleq c(s)\log(n)$ here we have 
\begin{align}
	\text{mmse}(\pi)
	&=\bbE[\|\vfbest(\vX) - \vtheta\|^2]
	\nonumber\\
	&=\bbE_{\vtheta\sim\pi}[\bbE_{\vX\sim\Poi(\vtheta)}[\|\vfbest(\vX) - \vtheta\|^2] | \vtheta]
	\nonumber\\
	&\ge \bbE_{\vtheta\sim\pi}[\bbE_{\vX\sim\Poi(\vtheta)}[\|\vfbest(\vX) - \vtheta\|^2] \indc{\vtheta \in [0, M]^d} | \vtheta]
	\nonumber\\
	&=\bbP[\vtheta \in [0, M]^d] \bbE_{\vtheta\sim\pi}[\bbE_{\vX\sim\Poi(\vtheta)}[\|\vfbest(\vX) - \vtheta\|^2] \indc{\vtheta \in [0, M]^d} | \vtheta]
	\nonumber\\
	&\ge (1 - \epsilon)\text{mmse}(\pi_{c, n})
\end{align}
and that $\text{mmse}(\pi)\le d$ given that the naive estimation of $\vf_{\mathsf{id}}(\vx)=\vx$ achieves an expected loss of $d$ (i.e. 1 for each coordinate). 
This shows that we also have $\mathsf{Regret}_{\pi}(\vfbest)\le \mathsf{Regret}_{\pi_{c, n}}(\vfbest) + O(\frac {d^2s^2}{n^2})\le \mathsf{Regret}_{\pi_{c, n}}(\vfbest) + O(\frac {ds^2}{n})$ in this multidimensional case (given that $d\le n$). 
Thus, it suffices to work on prior $\pi_{c, n}$ supported on $[0, c\log (n)]^d$ for some $c\triangleq c(s)$. 

Now under this truncated prior, by \prettyref{lmm:xmax-multidim} there exist absolute constants $c_3, c_4$ such that we may take $M = \max\{c_3, c_4s\}\log n$ and substitute into \prettyref{lmm:U1U2-bounds}. 
This gives an overall regret bound of $\frac dn{(\max\{c_3, c_4s\}\log (n))^{d+2}}$.

\bibliographystyle{alpha}
\bibliography{references}

\appendix

\section{Properties of Poisson mixtures}\label{app:poimix-props}

\begin{lemma}
	\label{lmm:bounded_prior}
	There exist constants $c_1, c_2$ such that for all $h > 0, k\ge 1$ and $\pi\in\mathcal{P}([0, h])$, 
	$X_{\max}$ on $n\ge 3$ samples have the following bound: 
	\begin{equation*}
		\bbP[1 + \max X_i\ge \max\{c_2, c_1h\}\cdot k\frac{\log n}{\log\log n}]\le n^{-k}
	\end{equation*}
	
\end{lemma}

\begin{proof}
	Consider $\lambda\in [0, h]$. Then for $x\ge h$ we have the following approximation for $X\sim \Poi(\lambda)$ via Chernoff's bound \cite[p.97-98]{mitzenmacher_upfal_2005}: 
	\begin{equation}\label{eq:bd-tail-bound}
		\bbP[X\ge x] \le \frac{(e\lambda)^x e^{-\lambda}}{x^x}\le \frac{(eh)^x e^{-h}}{x^x}
	\end{equation}
	Therefore for $X\sim p_{\pi}$ and $x\ge h$ we have $\bbP(X\ge x) \le\frac{(eh)^x e^{-h}}{x^x}$. 
	
	Now choose $c_0$ such that $c_0\ge \max\{4, h\}$, and for all $n\ge 3$, 
	\[
	\log \log n + \log c_0 - \log \log \log n - \log h - 1\ge \frac 12\log\log n
	\]
	That is, denoting $L = \sup_{n\ge 3} \sth{\log \log \log n - \frac 12 \log \log n}$, we take $\log c_0 \ge \log h + 1 + L$. 
	Notice that this mean we may take 
	$c_0 = \max\{4, \max \{1, \exp(1 + L)\}\cdot h\}$. 
	Then for all $k\ge 1$, $c_0k\frac{\log n}{\log \log n}\ge c_0\frac{\log n}{\log \log n}\ge c_0\ge h$ 
	given that $n > \log n$ for all $n > 1$, 
	so the tail bound in \prettyref{eq:bd-tail-bound} can be applied. 
	Setting $x = c_0k\frac{\log n}{\log \log n}$, we have 
	\begin{align}
		\log(\frac{(eh)^x e^{-h}}{x^x})
		&=-h + c_0k\frac{\log n}{\log \log n}(1 + \log h -\log c_0-\log k - \log \log n + \log \log \log n)
		\nonumber\\
		&\le -h + 4k\frac{\log n}{\log \log n}(-\frac 12 \log \log n)
		\nonumber\\
		& <  2k\log n\,, 
	\end{align}
	which implies that $\bbP[X\ge c_0k\frac{\log n}{\log \log n}] \le n^{-2k}$. 
	Finally, taking $c = 2c_0 = \max\{8, \max \{2, 2\exp(1 + L)\}\cdot h\}$, 
	we have 
	\[
	\bbP[1 + X_{\max}\ge ck\frac{\log n}{\log \log n}]\stepa{\le} n\bbP[1 + X\ge ck\frac{\log n}{\log \log n}]
	\stepb{\le} n\bbP[X\ge c_0k\frac{\log n}{\log \log n}]
	\stepc{\le} n^{-k}
	\]
	where (a) is union bound on $X_1, \cdots, X_n$, 
	(b) is using $\frac{\log n}{\log \log n} > 1$ for all $n\ge 3$ and 
	$\frac{\log n}{\log \log n}k(c - c_0)
	\ge c_0k\ge c_0 > 1$ for all $k\ge 1$, 
	and (c) is $2k-1\ge k$ for all $k\ge 1$. 
\end{proof}

\begin{lemma}
	\label{lmm:subexp_prior}
	There exist constants $c_1, c_2 > 0$ such that for all $s > 0, k\ge 1$ and $\pi\in\mathcal{P}([0, s\log n])$, 
	$X_{\max}$ on $n\ge 2$ samples has the following bound: 
	\begin{equation*}
		\bbP[X_{\max}\ge \max\{c_2, c_1s\}k\log n]\le n^{-k}
	\end{equation*}
\end{lemma}

\begin{proof}
	Again, consider the following argument via Chernoff's bound \cite[p.97-98]{mitzenmacher_upfal_2005}: for $x\ge s\log n$ and $X\sim p_{\pi}$ we have 
	\begin{equation*}
		\bbP[X\ge x]\le \sup_{0\le\lambda\le s\log n}\frac{(e\lambda)^{x}e^{-\lambda}}{x^x}
		\le \frac{(e s\log n)^{x}e^{-s\log n}}{x^x}
		=\exp(-s\log n + x(1 + \log (s\log n) - \log x))
	\end{equation*}
	Now, choose $c_0 = \max\{2 + s, e^2s\}$. Then for $k\ge 1$ and $x = kc_0\log n$ we have 
	\begin{align}
		&~-s\log n + (kc_0\log n)(1 + \log (s\log n) - \log(kc_0\log n))
		\nonumber\\
		=&~(\log n)(-s+kc_0(1 + \log s - \log k - \log c_0))
		\nonumber\\
		=& ~(\log n)(-s + kc_0(1-\log k - 2))
		\nonumber\\
		\le&~ (\log n)(-s - k(2 + s))
		\le ~(\log n)(-2k)
		\le ~(\log n)(-(k + 1))
	\end{align}
	Therefore $\bbP[X\ge c_0k\log n]\le n^{-(k+1)}$. 
	
	Take $c_3 = c_0(1 + \frac{1}{\log 2})$, we have 
	$1 + c_0k\log n\le c_3k\log n$ for all $k\ge 1$. 
	Therefore, union bound gives 
	$\bbP[1 + X_{\max}\ge c_3k\log n]\le n\bbP[1 + X\ge c_3k\log n]
	\le n\bbP[X\ge c_0k\log n]\le n^{-k}$. 
	It then follows that we can take $c_1 = e^2(1 + \frac{1}{\log 2})$ and $c_2 = 6(1 + \frac{1}{\log 2})$. 
\end{proof}

\begin{lemma}
	\label{lmm:exp_decay_lemma}
	Consider a random variable $W$. 
	If there exists a function $p(n)$ such that for all integers $c\ge 1$, $\bbP(W\ge cp(n))\le n^{-c}$, 
	then for each integer $m\ge 1$ there exists a constant $c(m)$ such that for all $n\ge 2$, 
	\[
	\bbE [W^m \indc{W\ge p(n)}]\le \pth{2^m+{3^m m!\over (\log n)^{m+1}}}\frac{p(n)^m}{n}
	\]
\end{lemma}

\begin{proof}[Proof of \prettyref{lmm:exp_decay_lemma}]

    Denote the event $E_k=\{kp(n)\le W\le (k + 1)p(n)\}$, then for all $n\ge 2$, we consider the expansion of $P(m, n)$ as per the claim to get 
	\begin{align}
		\bbE [W^m \indc{W\ge p(n)}]
		&= \sum_{k=1}^{\infty} \bbE[W^m \indc{E_k}]
		\le {(p(n))^m} \sum_{k=1}^{\infty} {(k+1)^m\over n^{k}}
        \leq {(p(n))^m\over n}\pth{{2^m} + 3^m \sum_{k=2}^{\infty} {(k-1)^m\over n^{k-1}}}
        \label{eq:m01}
	\end{align}
	Using the Gamma integration we bound the last term in the above display using
    \begin{align*}
        \sum_{k=2}^{\infty} {(k-1)^m\over n^{k-1}}
        \leq \int_{0}^\infty x^m n^{-x} dx
        =\int_{0}^\infty x^m e^{-x\log n} dx
        ={m!\over (\log n)^{m+1}}.
    \end{align*}
    Plugging this bound back in \eqref{eq:m01} finishes the proof.
\end{proof}

\begin{lemma}\label{lmm:xmax_bound}
    	Given $X_1, \cdots, X_n\simiid p_{\pi}\triangleq \Poi\circ \pi$. 
    	Let $k\ge 1$ be an integer. 
    	Then there exist constant $c_0(k), c_1, c_2, c_3, c_4$ such that:
    	\begin{itemize}
    		\item $\bbE[(1 + X_{\max})^k]\le c_0(k)(\max\{c_1, c_2h\}\frac{\log n}{\log \log n})^k$ for all $\pi\in \mathcal{P}([0, h])$. 
    		
    		\item $\bbE[(1 + X_{\max})^k]\le c_0(k)(\max\{c_3, c_4s\}\log n)^k$ for all $\pi\in\mathcal{P}([0, s\log n])$. 
    	\end{itemize}
    \end{lemma}

\begin{proof}
	For $\pi\in\mathcal{P}([0, h])$, 
	choose $c_1, c_2$ according to \prettyref{lmm:bounded_prior} and 
	use \prettyref{lmm:exp_decay_lemma} to obtain the constant $c_0(k)\triangleq (2^k + 2^k k!)$ with $p(n) \triangleq \max\{c_1, c_2h\}\frac{\log n}{\log \log n}$
	and $W = 1 + X_{\max}$. 
	For $\pi\in\mathcal{P}([0, s\log n])$, 
	choose $c_3, c_4$ according to \prettyref{lmm:subexp_prior} and 
	use \prettyref{lmm:exp_decay_lemma} with $p(n) \triangleq \max\{c_3, c_4s\}\log n$ and $W = 1 + X_{\max}$. 
\end{proof}

\begin{proof}[Proof of \prettyref{lmm:xmax-multidim}]
	We note that conditioned on $\theta_1, \cdots, \theta_d$, 
	the coordinates $X_1, \cdots, X_d$ are independent (distributed as $X_i\sim\Poi(\theta_i)$). 
	It then follows that 
	\[
	\bbE\qth{(1+X_{j, \max})^{\beta}\prod_{k=1\atop k\neq j}^d (1 + X_{k, \max})\mid \vtheta_1, \cdots, \vtheta_n}
	=\prod_{i=1}^d \bbE\qth{(1 + X_{i, \max})^{\beta_i} | \theta_{1i}, \cdots,\theta_{ni}}
	\]
	where here $\beta_i$ is $\beta$ if $i=j$ and 1 otherwise. 
	
	For the bounded prior case, i.e. $\pi\in \mathcal{P}([0, h])^d$ for some $h > 0$, 
	we may mimic the proof of \prettyref{lmm:bounded_prior} to obtain, for some absolute constant $c(h)\triangleq \max\{c_1, c_2h\}$, 
	$\bbP[1 + X_{i, \max} \ge kc(h)\frac{\log n}{\log \log n}\mid \theta_{1i}, \cdots, \theta_{ni}]\le n^{-k}$ 
	(given that $\theta\le h$). 
	Thus we may then adapt \prettyref{lmm:exp_decay_lemma} to yield 
	$\bbE[(1 + X_{i, \max})^{\beta_i} \mid \theta_{1i}, \cdots, \theta_{ni}]
	\le c_0(\beta_i)(c(h)\frac{\log n}{\log \log n})^{\beta_i}$ for some absolute constant $c_0(\beta_i)$ that depends only on the exponents $\beta_i$. 
	Since this inequality holds regardless of $\theta_{1i}, \cdots, \theta_{ni}$ 
	(so long as they are in the range $[0, h]$), 
	the desired bound now becomes 
	\[
	\bbE\qth{(1+X_{j, \max})^{\beta}\prod_{k=1\atop k\neq j}^d (1 + X_{k, \max})}
	\le c_0(\beta)c_0(1)^{d-1}\left(c(h)\frac{\log n}{\log \log n}\right)^{d - 1 + \beta}
	\]\[
	\le c_0(\beta)\left(c(h)\max\{1, c_0(1)\}\frac{\log n}{\log \log n}\right)^{d - 1 + \beta}
	\]
	
	Likewise, for the case $\pi\in \mathcal([0, s\log n]^d)$, 
	we may mimic the proof of \prettyref{lmm:subexp_prior} to obtain, 
	for some absolute constant $c'(s)\triangleq \max\{c_3, c_4h\}$, 
	$\bbP[1 + X_{i, \max}\ge kc(s)\log n \mid \theta_{1i}, \cdots, \theta_{ni}]\le n^{-k}$. 
	Using \prettyref{lmm:exp_decay_lemma} again, 
	$\bbE[(1 + X_{i, \max})^{\beta_i} \mid \theta_{1i}, \cdots, \theta_{ni}]
	\le c_0(\beta_i)(c'(s)\log n)^{\beta_i}$. Considering all $\vtheta_1, \cdots , \vtheta_n$ we then get 
	\[
	\bbE\qth{(1+X_{j, \max})^{\beta}\prod_{k=1\atop k\neq j}^d (1 + X_{k, \max})}
	\le c_0(\beta)\left(c'(s)\max\{1, c_0(1)\}\log n\right)^{d - 1 + \beta}
	\]

\end{proof}


\section{Proof of technical results}\label{app:technical}

\subsection*{Proof of \prettyref{lmm:erm_construction}}
Throughout the solution, for $s\le t$ we denote $m(s,t) \triangleq \frac{\sum_{i=s}^t w_i}{\sum_{i=s}^t v_i}$, 
where $m(s, t) =\infty$ if $v_i=0$ for $s\le i\le t$. 
Denote, also, the cost function $G(f)\triangleq \sum_{i=1}^n v_if(a_i)^2 - 2w_if(a_i)$. 
We restrict our attention to establishing $\ferm(a_1)$; 
the rest follows similarly. 
Let $i_2$ be the maximum index such that $\ferm(a_1) =\cdots = \ferm(a_{i_2})$ for some $i_2\ge 1$. 

We first claim that $\ferm(a_1) = m(1, a_{i_2})$. 
Indeed, for each real $t$, and integer $j=1, \cdots, k$, 
we define the following function $f_{j, t}(a_i) \triangleq 
\begin{cases}
    \ferm(a_i) + t & 1\le i\le j\\
    \ferm(a_i) & \text{otherwise}
\end{cases}$. 
Then by the maximality of $i_2$, for some small $\epsilon > 0$, 
$f_{i_2, t}$ is still monotone for some $t\in (-\epsilon, \epsilon)$. 
In addition, 
\begin{equation}
    \frac{\partial G(f_{j, t})}{\partial t} = \sum_{i=1}^j 2(v_i(\ferm(a_i)+t) - w_i)\,.
\end{equation}
Since $\ferm=\argmin G(f)$, 
$\frac{\partial G(f_{i_2, t})}{\partial t}|_{t=0}=0$. 
Therefore, 
\begin{equation}\label{eq:partial_eq}
\ferm(a_1)\sum_{i=1}^{i_2} v_i
=\sum_{i=1}^{i_2} \ferm(a_i)v_i = \sum_{i=1}^{i_2} w_i\,.
\end{equation}
Since $\max\{v_i, w_i\} > 0$ and each $v_i, w_i$ is nonnegative, 
we cannot have $\sum_{i=1}^{i_2} v_i=\sum_{i=1}^{i_2} w_i=0$. 
It then follows that $\ferm(a_1) = \frac{\sum_{i=1}^{i_2} w_i}{\sum_{i=1}^{i_2} v_i} = m(1, i_2)$. 

It now remains to show that $m(1, i_2)\le m(1, j)$ for all $j=1, \cdots, k$, 
and the inequality is strict for $j > i_2$. 
Now for any $j$ with $1\le j\le k$, for some small $\epsilon > 0$, 
$f_{j, t}$ is still monotone for some $t\in (-\epsilon, 0]$. 
Given also $\ferm=\argmin G(f)$, 
$\frac{\partial G(f_{j, t})}{\partial t}|_{t=0}\le 0$. 
Since $\ferm(a_i)\ge \ferm(a_1)$ for all $i$, 
we have 
\begin{equation}\label{eq:partial_sum}
\ferm(a_1)
\sum_{1\le i\le j} v_i 
\le 
\sum_{1\le i\le j} \ferm(a_i)v_i  \le \sum_{1\le i\le j} w_i\,,
\end{equation}
which implies that $m(1, j)\ge \ferm(a_1) = m(1, i_2)$. 
To show that $m(1, j) > m(1, i_2)$ for all $j > i_2$, 
suppose otherwise that $m(1, j) = m(1, i_2)$ for some $j > i_2$. 
This means the inequality in \prettyref{eq:partial_sum} is an equality for this $j$. 
In particular, 
\begin{equation}\label{eq:j_ineq}
    \ferm(a_1)
\sum_{i=1}^j v_i = \sum_{i=1}^j \ferm(a_i)v_i
\end{equation}
In view of \prettyref{eq:partial_eq}, 
from $\sum_{i=1}^j \ferm(a_i)v_i  = \sum_{i=1}^j w_i$ we have 
\begin{equation}
    \sum_{i=i_2+1}^j \ferm(a_i)v_i  = \sum_{i=i_2+1}^j w_i\,.
\end{equation}
By the maximality of $i_2$, we have $\ferm(a_i) > \ferm(a_1)$ for all $i > i_2$. 
Given that $v_i\ge 0$ for all $i$, 
\prettyref{eq:j_ineq} then implies $v_i=0$  for $i=i_2+1, \cdots, j$. 
This would imply that $\sum_{i=i_2+1}^j w_i=0$, 
i.e. $w_i = 0$ for all $i=i_2+1, \cdots, j$. 
This contradicts $\max\{v_i, w_i\} > 0$ for each $i=1, \cdots, n$.

\subsection*{Proof of \prettyref{lmm:binom_tail_rad}}\label{app:binom_tail_proof}
Recall that conditioned on $X_1^n$, 
$\epsilon(x)\sim 2\cdot Binom(N(x), \frac 12) - N(x)$. 
Since $b > 1$, it then follows that 
\begin{align*}
	\bbE[\max\{\epsilon(x) - \frac 1bN(x), 0\}]
	&= \bbE[(\epsilon(x) - \frac 1bN(x)) \indc{\epsilon(x) > \frac 1bN(x)}]
	\nonumber\\
	& \le (1-\frac 1b)N(x)\bbP[\epsilon(x) > \frac 1bN(x)]
	\nonumber\\
	& \stepa{\le} (1-\frac 1b)N(x)\exp(-N(x)D(\frac{1+\frac 1b}{2} || \frac 12))
	\stepb{\le} \frac{1-\frac 1b}{e\cdot D(\frac{1+\frac 1b}{2} || \frac 12)}
\end{align*}
where (a) is from \cite[Example 15.1, p.254]{polyanskiy_information_2022} and (b) is using the fact that for all $a > 0$ and $y\ge 0$, $y\exp(-ay)\le \frac{1}{ae}$.

	\subsection*{$O(X_{\max}\log X_{\max})$ Time Complexity Optimization}\label{app:stack}
	We now describe an algorithm based on stack that reduces the computation in \prettyref{lmm:erm_construction} from $O(X_{\max}^2)$ to  $O(X_{\max}\log X_{\max})$, 
	with this log factor only used in sorting $\{(X, N(X))\}$ for $X=0, 1, \cdots, X_{\max}$. 
	
	Let $W_1 < \cdots < W_k$ be the distinct elements in $\{X_1, \cdots, X_n\}\cup \{X_1 - 1, \cdots, X_n - 1\}$. 
	We consider a stack $S$, initialized as $\emptyset$, with each element being the triple $(I, w, t)$ where $I$ denotes the interval of piecewise constancy, $w=\sum_{k\in I} N(W_k)$ and $t = \sum_{j\in I}(W_k+1)N(W_k+1)$. 
	The invariant we are maintaining here is that the ratio $\frac{t}{w}$ is nondecreasing (this ratio is considered as $+\infty$ if $w=0$). 
	
	At each step $t=1, \cdots, k$ we do the following: 
	\begin{itemize}
		\item Initialize $a\triangleq ([t, t], N(W_t), (W_t + 1)N(W_t + 1))$, the active element; 
		
		\item 
		Suppose, now, $a=(I, w, t)$. 
		While the stack is nonempty and the top (most recent) element $a'=(I, w, t)$ $w't\le wt'$ (in particular, when $w, w' > 0$ we have the ratio $\frac{t}{w}\le \frac{t'}{w'}$), 
		we pop $a'$ from the stack, 
		and set $a = (I\cup I', w + w', t + t')$. 
		
		\item Push $a$ onto the stack. 
	\end{itemize}
    Then for each element in the form $([a, b], w, t)$ we have $\ferm(x) = \frac{t}{w}$ for all $x=W_a, \cdots, W_b$. 
    Notice that the largest element, $W_k$, has $N(W_k) > 0$, so the solution will always be well-formed. 
    
    To justify the time complexity, we see that there are at most $k$ pushes into the stack. 
    Each pop decreases the stack size by 1, so that cannot appear more than $k$ times either. 
    Assuming that each elementary computation (e.g. calculating $w't$ and $wt'$) is $O(1)$, 
    this stack operation takes $O(k)$. 
    Since $k\le X_{\max}$, the claim follows. 
	
\subsection*{Proof of \prettyref{lmm:binom_penal}}
	We will bound $\bbP[L_c(\epsilon) \ge k]$ for each integer $k\in [0, n]$. 
	First, we see that $\sum_{i=1}^j \epsilon_i - cj\le (1-c)j$ (i.e. we'll only consider $j \ge k$)
	and for this sum to be positive we need $\sum_{i=1}^j \epsilon_i > cj$. 
	If $X_j\sim Binom(j, \frac 12)$ we have 
	\[
	\bbP[\sum_{i=1}^j \epsilon_i> cj]
	=\bbP[X_j > j(\frac{c+1}{2})]
	\le \exp(-j D(\frac{c+1}{2} || \frac 12))
	\]
	by (i.e. \prettyref{lmm:binom_tail_rad}). 
	Now denoting $D(\frac{c+1}{2} || \frac 12)=c_1 > 0$, we have 
	\begin{align}
		\bbP[L_c(\epsilon) \ge k]
		&=\bbP[\exists j\ge k: \sum_{i=1}^j \epsilon_i - cj\ge k]
		\nonumber\\
		&\le \sum_{j=k}^n \bbP[\sum_{i=1}^j \epsilon_i - cj\ge k]
		\le \sum_{j=k}^n\exp(-j c_1)
		\le \frac{\exp(-c_1k)}{1-\exp(-c_1)}
	\end{align}
	Therefore we have 
	\begin{equation*}
		\EE[L_c(\epsilon)]
		\le 1 + \sum_{k=0}^n\bbP[L_c(\epsilon) \ge k]
		\le 1 + \sum_{k=0}^n \frac{\exp(-c_1k)}{1-\exp(-c_1)}
		\le 1 + \frac{1}{(1-\exp(-c_1))^2}. 
	\end{equation*}
    as desired.

\end{document}